\documentclass[11pt]{article}
\usepackage{mathrsfs}
\usepackage{graphicx}
\usepackage{epsfig}
\usepackage{epstopdf}
\usepackage{tikz}
\usepackage{geometry}

\usepackage{amsmath}
\usepackage{amsthm}
\usepackage{amssymb}
\usepackage{amsfonts}
\def\Im {\mathop{\rm Im}\nolimits}
\def\arg {\mathop{\rm arg}\nolimits}

\def\Ai {{\rm Ai}}

\def\tr {{\rm tr}}

\newtheorem{pro}{{\sc{Proposition}}}
\newtheorem{rem}{{\sc{Remark}}}
\newtheorem{lem}{{\sc{Lemma}}}
\newtheorem{cor}{{\sc{Corollary}}}
\newtheorem{thm}{{\sc{Theorem}}}

\numberwithin{equation}{section}

\begin{document}

\title{{ Gaussian unitary ensemble with jump discontinuities and the coupled Painlev\'e II and IV systems
   }}

\author{Xiao-Bo Wu\footnotemark[1] ~and Shuai-Xia Xu\footnotemark[2] ~}
\renewcommand{\thefootnote}{\fnsymbol{footnote}}
\footnotetext[1]{School of Mathematics and Computer Science, Shangrao Normal University, Shangrao  334001, China. E-mail: \texttt{wuxiaobo201207@163.com}}
\footnotetext[2]{Institut Franco-Chinois de l'Energie Nucl\'{e}aire, Sun Yat-sen University,
Guangzhou 510275, China. E-mail: \texttt{xushx3@mail.sysu.edu.cn}}

\date{}

\maketitle

\noindent \hrule width 6.27in \vskip .3cm

\noindent {\bf{Abstract }} We study the orthogonal polynomials and the Hankel determinants associated with Gaussian weight with two jump
discontinuities. When the degree $n$ is finite, the orthogonal polynomials and the  Hankel determinants are shown to be connected to the coupled Painlev\'e IV system. In the double scaling limit as the jump discontinuities tend to the edge of the spectrum and the degree $n$ grows to infinity, we establish the asymptotic expansions for the Hankel determinants and the orthogonal polynomials, which are expressed in terms of solutions of the coupled Painlev\'{e} II system.  As applications, we re-derive the recently found Tracy-Widom type expressions for the gap probability of there being no eigenvalues in a finite interval near the the extreme eigenvalue of large Gaussian unitary ensemble and the limiting conditional distribution of the  largest eigenvalue in  Gaussian unitary ensemble by considering a thinned process.


\vspace{.2in}


\vskip .3cm

\noindent \hrule width 6.27in\vskip 1.3cm

\tableofcontents

\section{Introduction and statement of results}

Consider the Gaussian Unitary Ensemble (GUE), where  the joint probability density function of the eigenvalues is given by
\begin{equation}\label{GUE}\rho_n(\lambda_1,...,\lambda_n)=\frac{1}{Z_{n}}\prod_{i=1}^{n}e^{-\lambda_i^2}\prod_{1\leq i<j\leq n}^{n}|\lambda_i-\lambda_j|^2; \end{equation}
see \cite{ Forrester, Metha}. Here  $Z_{n}$, known as the partition function, is a normalization constant.
It is well known that  the density function \eqref{GUE} can be expressed in the determinantal form
\begin{equation}\label{determiantal form} \rho_n(\lambda_1,...,\lambda_n)=\det[K_n(\lambda_i,\lambda_j)]_{1\leq i,j\leq n},\end{equation}
where
\begin{equation}\label{Weighted OP kernel}
K_n(x,y)=e^{-\frac{1}{2}(x^2+y^2)}\sum_{k=0}^{n-1}H_k(x)H_k(y).
\end{equation}
The polynomial $H_k(x)$ therein is the $k$-th degree  monic orthogonal polynomial with respect to the Gaussian weight $e^{-x^2}$, which is  the Hermite polynomial except for a constant \cite[Eq. (18.5.13)]{O}.

Introduce the Hankel determinant
\begin{align}\label{def: Hankel}
D_n(s_1,s_2; \omega_1,\omega_2)&=\det\left(\int_{\mathtt{R}} x^{j+k} w(x;s_1,s_2;\omega_1,\omega_2)dx\right)_{j,k=0}^{n-1}\nonumber \\
&=\frac{1}{n!}\int_{-\infty}^{\infty}\cdots \int_{-\infty}^{\infty} \prod_{i=1}^{n}w(x_j;s_1,s_2;\omega_1,\omega_2)\prod_{1\leq i<j\leq n}^{n}(x_i-x_j)^2dx_1 \cdots dx_n,
                                                       \end{align}
where
\begin{equation}\label{weight}
w(x;s_1,s_2;\omega_1,\omega_2)=e^{-x^2}\left\{\begin{array}{cc}
                                                          1& x< s_1,\\
                                                           \omega_1& s_1<x<s_2,\\
                                                        \omega_2& x>s_2,
                                                       \end{array}
                                                       \right.
                                                       \end{equation}
with the constant $ \omega_k\geqslant0, k=1,2$.
If $\omega_1=\omega_2=1$, the Hankel determinant $D_n(s_1,s_2;1,1)$ is corresponding to the pure Gaussian weight $e^{-x^2}$ and can be evaluated  explicitly
 \begin{equation}\label{def: Hankel-GUE}
D_n^{\texttt{GUE}}=D_n(s_1,s_2;1,1)=(2\pi)^{n/2}2^{-n^2/2}\prod_{k=1}^{n-1}k! ~;
                                                       \end{equation}
 see \cite[Equation (4.1.5)]{Metha}.
There is a system of  monic orthogonal polynomials $\pi_n(x)=\pi_n(x;s_1,s_2)=x^n+\cdots$,
 orthogonal with respect to the weight function $w(x)=w(x;s_1,s_2;\omega_1,\omega_2)$,
 \begin{equation}\label{def: Orth}
         \int_{\mathbb{R}}\pi_m(x)\pi_n(x) w(x)dx  =  \gamma_n^{-2}\delta_{m,n}, \quad m, n\in\mathbb{N}.
          \end{equation}
           The orthogonal polynomials satisfy the three term recurrence relation
   \begin{equation}\label{def: rec}
       z\pi_n(z)=\pi_{n+1}(z)+   \alpha_n\pi_{n}(z)+\beta_n^2\pi_{n-1}(z),
          \end{equation}
   where $\alpha_n=\alpha_n(s_1,s_2)$    and $\beta_n=\beta_n(s_1,s_2)$ are the recurrence coefficients.
The polynomials $\gamma_{n}\pi_n(x)$ are the normalized orthogonal polynomials and the leading coefficients $\gamma_{n}=\gamma_n(s_1,s_2)$ are connected to the Hankel determinant $D_n(s_1,s_2)=D_n(s_1,s_2; \omega_1,\omega_2)$ by
\begin{equation}\label{eq:DGamma}D_n(s_1,s_2)=\prod_{j=0}^{n-1}\gamma_j^{-2}(s_1,s_2).\end{equation}

 Consider the gap probability that there is no eigenvalues  in the finite interval $(s_1,s_2)$ for the GUE matrices.  On account of \eqref{GUE} and \eqref{def: Hankel}, the gap probability can be expressed as a ratio of the Hankel determinants
\begin{equation}\label{def: Gap-pro}
\texttt{Pro}(\lambda_j\not\in(s_1,s_2): j=1...n)=\frac{D_n(s_1,s_2; 0,1)}{D_n^{\texttt{GUE}}},
                                                       \end{equation}
 where  $\lambda_1<, ...,\lambda_n$ are the eigenvalues  of a matrix in GUE and  $D_n(s_1,s_2;\omega_1,\omega_2)$ is defined in \eqref{def: Hankel}.
For  the gap probability on the  infinite interval $(s, +\infty)$, we have the distribution of the largest eigenvalue
 \begin{equation}\label{def: DistLargest}
\texttt{Pro}(\lambda_n<s)=\frac{D_n(s,s; 0,0)}{D_n^{\texttt{GUE}}}.
                                                       \end{equation}
In the large  $n$ limit,  the distribution of the largest eigenvalue converges to the celebrated  Tracy-Widom distribution
 \begin{equation}\label{eq:TW}
\lim_{n\to+\infty}\texttt{Pro}(\lambda_n<\sqrt{2n}+\frac{s}{\sqrt{2}n^{1/6}})=\exp\left(-\int_s^{+\infty}(x-s)q_{\texttt{HM}}^2(x)dx\right),
                                                       \end{equation}
 where $q_{\texttt{HM}}(s)$ is the Hastings-Mcleod solution of the second Painlev\'e equation $q''(x)-2q^3(x)-xq(x)=0$ with the asymptotic behavior $q_{\texttt{HM}}(x)\sim \Ai (x)$ as $x\to+\infty$; see \cite{TW}.

We proceed to consider the thinned process in GUE  by removing each eigenvalues $\lambda_1<, ...,\lambda_n$ of the GUE independently with probability $p\in(0,1)$; see \cite{BP04, BP06}. It is observed in \cite{BP06} that the remaining and removed eigenvalues can be interpreted as observed and unobserved particles, respectively.  If we know the information that the largest observed particle
$\lambda^{T}_{\max}$ is less that $y$,
then the conditional distribution of the largest eigenvalue $\lambda_n$ of the  original GUE can be expressed by the ratio of Hankel determiants
\begin{equation}\label{def: ConditionalGap-pro}
\texttt{Pro}(\lambda_n<x| \lambda^{T}_{\max}<y)=\frac{D_n(y,x;p,0)}{D_n(y,x;p,p)},\quad x>y,
                                                       \end{equation}
                                                       and
                                                       \begin{equation}\label{def: ConditionalGap-pro-2}
\texttt{Pro}(\lambda_n<x| \lambda^{T}_{\max}<y)=\frac{D_n(x,y;0,0)}{D_n(x,y;p,p)},\quad x<y,
                                                       \end{equation}
with  $D_n(s_1,s_2;\omega_1,\omega_2)$  defined in \eqref{def: Hankel};  see \cite{cc,cd} and \cite{BC, cd-1}.  It is noted that other thinned  random matrices in the situation of circular ensemble is considered in  \cite{cc} and  also in  \cite{bf} with applications in  the studies of  Riemann zeros.

Recently, the limits of \eqref{def: Gap-pro} and  \eqref{def: ConditionalGap-pro} are studied in \cite{cc-19,cd} by considering the Fredholm determinants of the  Airy  kernel with several discontinuities. More generally, the limits of the gap probabilities on any finite union of
intervals near the extreme eigenvalues are considered in \cite{cd}.  In \cite{XuDai2019}, the second author of the present paper and Dai derive the asymptotics of   \eqref{def: ConditionalGap-pro} via the  Fredholm determinants of the Painlev\'e XXXIV kernel which is a generalization of the Airy kernel.  In both \cite{cd} and \cite{XuDai2019},  Tracy-Widom type expressions for the limiting  distributions are established by using solutions to  the  coupled Painlev\'e II system. The Hankel determinants and orthogonal polynomials associated with the Gaussian weight with one jump discontinuity have also been considered in \cite{BC, FW, IK, MC, WuXuZhao, XuZhao}  with applications in random matrices.

The present work is devoted to the studies of  the Hankel determinants and the orthogonal polynomials associated with the Gaussian weight
with two jump discontinuities both as the degree $n$ is finite and as $n$ tends to infinity.   When the degree $n$  is finite, we show that the Hankel determinants and  the orthogonal polynomials are described by the coupled Painlev\'e IV system. As the jump discontinuities tend to the largest eigenvalue of GUE and the degree $n$ grows to infinity, we establish   asymptotic expansions  for the Hankel determinants and  the orthogonal polynomials.  The asymptotics are expressed in terms of   solutions to  the  coupled Painlev\'e II system. As applications, our results reproduce  the asymptotic expansions of the  gap probability in a finite interval near the largest eigenvalue of GUE and the conditional  distribution of largest eigenvalue of GUE as defined in  \eqref{def: Gap-pro} and \eqref{def: ConditionalGap-pro}, respectively,  which are obtained previously in \cite{cd, XuDai2019}.

\subsection{Statement of results}
\vskip .3cm
\subsection*{The coupled Painlev\'{e} IV system }
We introduce the Hamiltonian
\begin{equation}\label{int: H}
H_{\texttt{IV}}(a_1,a_2,b_1,b_2; x; s)=-2(a_1b_1+a_2b_2+n)(a_1+a_2)+2(a_1b_1(x-s)+a_2b_2(x+s)+nx)-(a_1b_1^2+a_2b_2^2),
  \end{equation}
  which is a special Garnier system in two variables in the studies of the classification of 4-dimensional Painlev\'e-type equations by Kawakami, Nakamura and Sakai \cite[Equations (3.12)-(3.13)]{KawNakSak}.
The coupled Painlev\'e IV system can be written as  the following Hamiltonian system
  \begin{equation}\label{int: HIV-system}
   \left\{\begin{array}{l}
\frac{d a_1}{d x}=\frac{\partial H_{\texttt{IV}}}{\partial b_1}(a_1,a_2,b_1,b_2; x,s)=-2a_1(a_1+a_2+b_1-x+s),\\
\frac{d a_2}{dx}=\frac{\partial H_{\texttt{IV}}}{\partial b_2}(a_1,a_2,b_1,b_2; x,s)=-2a_2(a_1+a_2+b_2-x-s),\\
\frac{d b_1}{d x}=-\frac{\partial H_{\texttt{IV}}}{\partial a_1}(a_1,a_2,b_1,b_2; x,s)=b_1^2+2b_1(2a_1+a_2-x+s)+2(a_2b_2+n),\\
\frac{db_2}{dx}=-\frac{\partial H_{\texttt{IV}}}{\partial a_2}(a_1,a_2,b_1,b_2; x,s)=b_2^2+2b_2(a_1+2a_2-x-s)+2(a_1b_1+n).
\end{array}\right.
\end{equation}
Eliminating $b_1$ and $b_2$ from the system, we  find that  $a_1$ and $a_2$ solve  a couple of second order nonlinear differential equations
\begin{equation}\label{int:ak}
\left\{\begin{array}{l}
       \frac{\mathtt{d} ^2a_1 }{\mathtt{d} x^2}-\frac 1{2a_1} \left(\frac{\mathtt{d} a_1}{\mathtt{d} x}\right)^2-6a_1(a_1+a_2)^2+8a_1(a_1+a_2)x-8a_1^2s+2(2n-1)a_1-2a_1(x-s)^2=0,\\
         \frac{\mathtt{d}^2a_2}{\mathtt{d} x^2}-\frac 1{2a_2} \left(\frac{\mathtt{d} a_2}{\mathtt{d} x}\right)^2-6a_2(a_1+a_2)^2+8a_2(a_1+a_2)x+8a_2^2s+2(2n-1)a_2-2a_2(x+s)^2=0.
       \end{array}
\right.
\end{equation}
    If $a_2=0$,   we recover from the above equations the classical  Painlev\'{e} IV equation (see \cite{FokasBook} and \cite[Equation (32.2.4)]{O}  )
 \begin{equation}\label{eq:PIV}
  y_{\texttt{IV}}''=\frac{1}{2y_{\texttt{IV}}}y_{\texttt{IV}}'^2+\frac{3}{2}y_{\texttt{IV}}^3+4x y_{\texttt{IV}}^2+2(x^2+1-2n)y_{\texttt{IV}}, \quad y_{\texttt{IV}}(x)=-2a_1(x+s;s).
\end{equation}

\subsection*{Orthogonal polynomials of finite degree: the coupled  Painlev\'{e} IV system  }
Our first result shows that, when the degree $n$ is finite,  several quantities of the orthogonal polynomials associated with the weight function \eqref{weight} can be expressed in terms
of the coupled  Painlev\'{e} IV system. These quantities include the
the Hankel determiniants,  the recurrence coefficients,  leading coefficients and the values of the orthogonal polynomials
at the jump discontinuities  of  \eqref{weight}.  We are interested in the Gaussian weight with two jump discontinuities, thus without loss of generality we assume that the parameters in \eqref{weight} satisfy
 \begin{equation}\label{def:Para}
s_1<s_2;  \quad \omega_1\geqslant0 , \quad  \omega_2\geqslant0,  \quad \omega_1\neq \omega_2, \quad \omega_1\neq 1.
 \end{equation}

\begin{thm} \label{thm:HankelFixedn}
 Let   $s_k$ and $\omega_k$, $k=1,2$ be as in \eqref{def:Para} and $D_n(s_1,s_2)=D_n(s_1,s_2,\omega_1,\omega_2)$ be the Hankel determinant defined in \eqref{def: Hankel}, we denote
 \begin{equation}\label{def:F}
 F(s_1, s_2)=\frac{\partial}{\partial s_1}\ln D_n(s_1, s_2)+\frac{\partial}{\partial s_2}\ln  D_n(s_1, s_2),
 \end{equation}
 and
 \begin{equation}\label{def:x}
 x=\frac{s_1+s_2}{2}, \quad s=\frac{s_2-s_1}{2}.
 \end{equation}
Then $F(s_1,s_2)$  is related to the Hamiltonian for the coupled Painlev\'e IV system by
\begin{equation}\label{thm: F-H}
 F(s_1,s_2)=H_{\texttt{IV}}(x;s)-2nx.
  \end{equation}
Moreover, let $\alpha_n(s_1,s_2)$,  $\beta_n(s_1,s_2)$ be the recurrence coefficients defined in \eqref{def: rec},  $\gamma_n(s_1,s_2)$ be the leading coefficient of the orthonormal polynomial defined in \eqref{def: Orth} and $\pi_n(x)=\pi_n(x;s_1,s_2)$ be the  monic orthogonal polynomial defined in \eqref{def: Orth}, we have    \begin{align}\label{thm:alpha}
&\alpha_n(s_1,s_2)=\frac{a_1(x;s)b_1^2(x;s)+a_2(x;s)b_2^2(x;s)}{2(a_1(x;s)b_1(x;s)+a_2(x;s)b_2(x;s)+n)}, \\ \label{thm: beta}
& \beta^2_n(s_1,s_2)=\frac {1}{2}\left(a_1(x;s)b_1(x;s)+a_2(x;s)b_2(x;s)+n\right),\\ \label{thm: gamma-0}
&\gamma_{n-1}^2=\frac{1}{4\pi i}e^{x^2}y(x;s)\neq 0,\\ \label{thm: gamma}
&\frac{d }{dx}\ln \gamma_{n-1}(s_1,s_2)=a_1(x;s)+a_2(x;s),\\  \label{thm:pns1}
&\pi_n(s_1)^2=\frac{2\pi i }{\omega_1-1}e^{-2sx+s^2}\frac{a_1(x;s)b_1(x;s)^2}{y(x;s)},\\  \label{thm:pns2}
&\pi_n(s_2)^2=\frac{2\pi i}{\omega_2-\omega_1}e^{2sx+s^2}\frac{ a_2(x;s)b_2(x;s)^2}{y(x;s)}, \end{align}
where $a_k(s_1,s_2)$ and $b_k(s_1,s_2)$, $k=1,2$,  satisfy the coupled Painlev\'e IV system \eqref{int: HIV-system} and
$y(x;s)$ is connected to $a_k(x;s)$, $k=1,2$, by $\frac{d y}{d x}=2(a_1+a_2-x)y$.
     \end{thm}
\begin{rem} \label{rem:PIV}
In view of  \eqref{thm: gamma-0} and
\eqref{thm:pns2}, we have
\begin{equation}\label{eq: akbk-est}
a_2(x;s)b_2(x;s)^2=O(\omega_1-\omega_2), \quad \mbox{as}  \quad \omega_1\to \omega_2,
\end{equation}
where the error bound is uniform for $x$ and $s$ in any compact subset of $\mathbb{R} $ and $\mathbb{R}\setminus\{0\}$, respectively.
Using $\pi_{n-1}(s_2)^2=\frac{8\pi i}{\omega_2-\omega_1}e^{2sx+s^2}\frac{ a_2(x;s)}{y(x;s)}$ (see \eqref{eq:qij}), we obtain that
\begin{equation}\label{eq: ak-est}
a_2(x;s)=O(\omega_1-\omega_2), \quad \mbox{as}  \quad \omega_1\to \omega_2.
\end{equation}
Since $a_2(x;s)\to 0$, as $\omega_1\to \omega_2$,  the function $y_{\texttt{IV}}(x)=-2a_1(x+s;s)$ solves the classical Painlev\'e IV equation as shown before in \eqref{eq:PIV}.
Thus, as $\omega_1\to \omega_2$, Theorem \ref{thm:HankelFixedn} implies that the Hankel determinants and the orthogonal polynomials associated with the weight function \eqref{weight} with one discontinuity are related to the classical  Painlev\'e IV equation.
\end{rem}

\subsection*{The coupled Painlev\'{e} II system}

To state our main results on the asymptotics of the orthogonal polynomials, we  introduce the following coupled  Painlev\'{e} II system in dimension four
\renewcommand{\arraystretch}{1.3}
\begin{equation}\label{eq:CPII}
\left\{\begin{array}{l}
         \frac {dw_1}{dx}=-\frac{\partial H_{\texttt{II}}}{\partial v_1} =2(v_1+v_2+\frac{x}{2})-w_1^2,\\
         \frac {dv_1}{dx}=\frac{\partial H_{\texttt{II}}}{\partial w_1} = 2v_1w_1,\\
          \frac {dw_2}{dx}=-\frac{\partial H_{\texttt{II}}}{\partial v_2} =2(v_1+v_2+\frac{x+s}{2})-w_2^2, \\
         \frac {dv_2}{dx}=\frac{\partial H_{\texttt{II}}}{\partial w_2} = 2v_2w_2,\\
       \end{array}
\right.
\end{equation}
where $v_k=v_k (x;s)$, $w_k=w_k (x;s)$, $k=1,2$ and  the Hamiltonian $H_{\texttt{II}}=H_{\texttt{II}}(v_1,v_2,w_1,w_2;x; s)$ is given by
\begin{equation}\label{def:Hamiltonian-CPII}
H_{\texttt{II}}(v_1,v_2,w_1,w_2;x;s)=-(v_1+v_2)^2-(v_1+v_2)x+v_1w_1^2+v_2w_2^2-sv_2.
\end{equation}
The coupled  Painlev\'{e} II system
appears in both of the degeneration schemes of the
Garnier system  in two variables \cite[Equations (3.5)-(3.7))]{Kaw2017} and  the Sasano system \cite[Equations (3.22)-(3.23)]{Kaw2018} by Kawakami.

Eliminating $w_1$ and  $w_2$ from the Hamiltonian system \eqref{eq:CPII} gives us the following nonlinear equations for $v_1$ and $v_2$
\begin{equation}\label{int-equation v}
\left\{\begin{array}{l}
         v_{1xx}-\frac{v_{1x}^2}{2v_1} -4v_1(v_1+v_2+\frac{x}{2})=0, \\
          v_{2xx}-\frac{v_{2x}^2}{2v_2} -4v_2(v_1+v_2+\frac{x+s}{2})=0.\\
       \end{array}
\right.
\end{equation}
Let $v_k(x;s)=u_k(x;s)^2=u_k(x)^2$ , $k=1,2$, the above equations are further simplified to
\begin{equation}\label{int-equation u}
\left\{\begin{array}{l}
         u_{1xx}-xu_1-2u_1(u_1^2+u_2^2)=0, \\
          u_{2xx}-(x+s)u_2 -2u_2(u_1^2+u_2^2)=0.\\
       \end{array}
\right.
\end{equation}
If $v_2(x)=u_2(x)^2=0$,  then \eqref{int-equation u} is
reduced to the classical second Painlev\'e equation
\begin{equation}\label{eq:pII}
q''-2q^3-xq=0.
\end{equation}
The functions $v_k(x;s)$ and $u_k(x;s)$, $k=1,2$,  are also connected to  $H_{\texttt{II}}(x;s)=H_{\texttt{II}}(v_1,v_2,w_1,w_2;x;s)$ by
\begin{equation}\label{eq: DHaml}
\frac{d}{dx} H(x;s)=-(v_1(x;s)+v_2(x;s))=-(u_1(x;s)^2+u_2(x;s)^2),\end{equation}
which can be obtained be taking derivative on both side of \eqref{def:Hamiltonian-CPII}; see also \cite[Equation (7.37)]{XuDai2019}.
The existence of   solutions to the coupled Painlev\'e II system are established in  \cite{cd, XuDai2019}.

\begin{pro} \label{pro:H} (\cite{cd, XuDai2019})
 For the parameters $\omega_k$, $k=1,2$ as given in \eqref{def:Para} and $ s> 0$,  there exist real-valued and pole-free solutions $v_k(x; s) $ (or $u_k(x; s)$), $k=1,2$,  to the coupled nonlinear differential equations \eqref{int-equation v} (or \eqref{int-equation u}) subject to the boundary conditions as $x\to+\infty$
 \begin{equation}\label{eq:u-asy}
v_1(x;s)=u_1(x;s)^2\sim (\omega_1-\omega_2) \Ai(x)^2, \quad v_2(x;s)=u_2(x;s)^2\sim (1-\omega_1) \Ai(x+s)^2,
  \end{equation}
  where $\Ai$ is the standard Airy function.
        \end{pro}



\subsection*{Asymptotics of the Hankel determinants and applications in random matrices }
Our second result gives  the asymptotics of the Hankel determinants expressed in terms of the solutions to the coupled Painlev\'e II system.
\begin{thm} \label{thm:HankelAsy}
  Let   $s_k$ and $\omega_k$, $k=1,2$ be as in \eqref{def:Para} and $s_k$ are related to $t_k$, $k=1,2$ by
 $$s_1=\sqrt{2n}+\frac {t_1}{\sqrt{2}n^{1/6}}, \quad s_2=\sqrt{2n}+\frac {t_2}{\sqrt{2}n^{1/6}}, $$
with $t_1<t_2$, then we have the asymptotics of  the Hankel determinant $D_n(s_1,s_2)=D_n(s_1,s_2;\omega_1,\omega_2)$ defined in  \eqref{def: Hankel}
as $n\to\infty$
\begin{equation}\label{thm: HAsy}
D_n(s_1,s_2)= D_n^{\texttt{GUE}}\exp\left(-\int_{t_1}^{+\infty}(\tau-t_1)(u_1(\tau;t_2-t_1)^2+u_2(\tau;t_2-t_1)^2)d\tau\right) \left(1+O(n^{-1/6})\right),\end{equation}
where $D_n^{\texttt{GUE}}$ is the Hankel determinant associated with the  Gaussian weight  with expression given in  \eqref{def: Hankel-GUE}, $u_k(x; s)$, $k=1,2$, are solutions to  \eqref{int-equation u}  subject to the boundary conditions \eqref{eq:u-asy}
and the error bound is uniform for $t_1$, $t_2$ in any compact subset of  $\mathbb{R}$.
\end{thm}



\begin{rem} \label{rem:HankelAsy}
When $s\to 0$, it is shown in \cite[Equation (1.28)]{cd} that
\begin{equation*}
u_1(x;s)^2+u_2(x;s)^2=q^2(x;\omega_2)+O(s),
 \end{equation*}
where $q(x;\omega_2)$ is the Ablowitz-Segur solution to the second Painlev\'e equation \eqref{eq:pII} with the  boundary condition  as $ x\to +\infty$
\begin{equation}\label{eq: ASAsy} q(x;\omega_2)\sim \sqrt{1-\omega_2}~\Ai(x).\end{equation}
Therefore, as $t_2-t_1\to 0$, the formula \eqref{thm: HAsy} is reduced to
\begin{equation}\label{eq: HAsy-reduced}
D_n(s_1)= D_n^{\texttt{GUE}}\exp\left(-\int_{t_1}^{+\infty}(\tau-t_1)q^2(\tau;\omega_2))d\tau\right) \left(1+O(n^{-1/6})\right),
 \end{equation}
 where $D_n(s_1)$ is the Hankel determinant associated with the weigh function \eqref{weight} with one jump discontinuity by  taking $s_1=s_2=\sqrt{2n}+\frac {t_1}{\sqrt{2}n^{1/6}}$.  The expansion  agrees with the result from \cite{BC} where the case with one jump discontinuity is considered.
\end{rem}


As an application of the asymptotics of the Hankel determinants.  We derive the  gap probability of there being no eigenvalues in a finite interval near the extreme eigenvalues of large GUE by using \eqref{def: Gap-pro} and \eqref{thm: HAsy}, which confirms a recent result from \cite[Equation (2.6)]{cd}.

\begin{cor} (\cite{cd}) \label{thm:GapPro}
 Let $s_1$ and $s_2$ be as in Theorem \ref{thm:HankelAsy}, we have the asymptotic approximation of  the gap probability of finding no eigenvalues of GUE in the finite interval  $(s_1,s_2)$
\begin{align}\label{thm: GapPro}
\texttt{Pro}(\lambda_j\not\in(s_1,s_2): j=1...n)&=\exp\left(-\int_{t_1}^{+\infty}(\tau-t_1)(u_1(\tau;t_2-t_1)^2+u_2(\tau;t_2-t_1)^2)d\tau\right) \nonumber\\
&~~~\times\left(1+O(n^{-1/6})\right),
 \end{align}
 where $u_k(x; s)$, $k=1,2$, are solutions to  \eqref{int-equation u}  subject to the boundary conditions \eqref{eq:u-asy}
with the parameters $\omega_1=0, \omega_2=1$ and
  the error bound is uniform for $t_1$ and $t_2$ in any compact subset of $\mathbb{R}$.
\end{cor}


In the second application, we derive from \eqref{thm: HAsy} and \eqref{eq: HAsy-reduced} the large $n$ limit of the  distribution \eqref{def: ConditionalGap-pro} in the thinning and conditioning GUE. This reproduces the result in \cite{cd, XuDai2019}.
\begin{cor} (\cite{cd, XuDai2019})
 Let $s_1$ and $s_2$ be as in Theorem \ref{thm:HankelAsy},
we have the asymptotics of  the conditional gap probability
\begin{align}\label{thm: ConGapPro}
\texttt{Pro}(\lambda_n<s_2| \lambda^{T}_{\max}<s_1)&=\exp\left(-\int_{t_1}^{+\infty}(\tau-t_1)(u_1(\tau;t_2-t_1)^2+u_2(\tau;t_2-t_1)^2-q^2(\tau;p))d\tau\right)\nonumber\\
&~~\times \left(1+O(n^{-1/6})\right),
 \end{align}
 where $u_k(x; s)$, $k=1,2$, are solutions to  \eqref{int-equation u}  subject to the boundary conditions \eqref{eq:u-asy}
with the parameters $\omega_1=p\in(0,1), \omega_2=0$,   $q(x;p)$   is the Ablowitz-Segur solution to the second Painlev\'e equation \eqref{eq:pII}   with the asymptotics \eqref{eq: ASAsy} and the error bound is uniform for $t_1$ and $t_2$ in any compact subset of $\mathbb{R}$.
\end{cor}

\subsection*{Asymptotics of the coupled Painlev\'{e} IV system}
Next, we show that  the scaling limit of the coupled Painlev\'e IV system leads to the  coupled Painlev\'e II system.
\begin{thm} \label{thm:CPIVAsy}
  Let $s_k$, $\omega_k$, $k=1,2$ be as in Theorem \ref{thm:HankelAsy} and $x=(s_1+s_2)/2$, $s=(s_2-s_1)/2$,
then we have the asymptotics of the coupled Painlev\'e IV system as $n\to\infty$
\begin{gather}\label{}
a_1(x;s)=-\frac{1}{\sqrt{2}n^{1/6}}\left(v_1(t_1;t_2-t_1)+
\frac{v_{1x}(t_1;t_2-t_1)}{2n^{1/3}}
+O(n^{-2/3})\right),\label{asy-cpiv-a-1}\\
a_2(x;s)=-\frac{1}{\sqrt{2}n^{1/6}}\left(v_2(t_1;t_2-t_1)+
\frac{v_{2x}(t_1;t_2-t_1)}{2n^{1/3}}
+O(n^{-2/3})\right),\label{asy-cpiv-a-2}\\
 b_1(x;s)=\sqrt{2n}\left(1-\frac{v_{1x}(t_1;t_2-t_1)}{2v_{1}(t_1;t_2-t_1)n^{1/3}}+O(n^{-2/3})\right),\label{asy-cpiv-b-1}\\
  b_2(x;s)=\sqrt{2n}\left(1-\frac{v_{2x}(t_1;t_2-t_1)}{2v_{2}(t_1;t_2-t_1)n^{1/3}}+O(n^{-2/3})\right),\label{asy-cpiv-b-2}\\
  y(x;s)=2i(2/n)^{n-\frac 12}e^{-n-(t_1+t_2)n^{1/3}}\left(1+\frac{H_{\texttt{II}}(t_1;t_2-t_1)-\frac{1}{8}(t_1+t_2)^2}{n^{1/3}}
+O(n^{-2/3})\right),\label{asy-cpiv-y}
\end{gather}
where  $v_1(x;s)$ and $v_2(x;s)$ are solutions to the   coupled Painlev\'e II system  \eqref{int-equation v}  subject to the boundary conditions \eqref{eq:u-asy}, $H_{II}(x;s)$ is the Hamiltonian associated to these solutions   and the
subscript $x$ in $v_{kx}(x;s)$ denotes the derivative of $v_k(x;s)$ with respect to $x$ for  $k=1,2$.
\end{thm}

\subsection*{Asymptotics of the orthogonal polynomials }
Applying Theorem \ref{thm:HankelFixedn} and \ref{thm:CPIVAsy}, we  obtain the asymptotics of the recurrence coefficients, leading coefficients of the orthogonal polynomials and the
values of the orthogonal polynomials at $s_1$ and $s_2$.

\begin{thm}\label{thm:OpAsy}
  Let $s_k$, $\omega_k$, $k=1,2$ be as in Theorem \ref{thm:HankelAsy},
 we have the asymptotics of the recurrence coefficients and leading coefficients of the orthogonal polynomials as $n\to\infty$
  \begin{gather}
  \alpha_n=-\frac {1}{\sqrt{2}}\left(v_1(t_1;t_2-t_1)+v_2(t_1;t_2-t_1)\right)n^{-1/6}+O(n^{-1/2}),\label{an-asymptotic-expansion}\\
   \beta_n=\frac{1}{\sqrt{2}}n^{1/2}-2^{-3/2}\left(v_1(t_1;t_2-t_1)+v_2(t_1;t_2-t_1)\right)n^{-1/6}+O(n^{-1/2}),\label{bn-asymptotic-expansion}\\
\gamma_{n-1}=2^{\frac n2-\frac 34}n^{\frac 14-\frac n2}e^{\frac n2}\pi ^{-1/2}\left(1+\frac{1}{2}H_{\texttt{II}}(t_1;t_2-t_1)n^{-1/3}+O(n^{-2/3})\right).
\end{gather}
Moreover, we derive the asymptotics of  the
values of the orthogonal polynomials at $s_k$  as $n\to\infty$
\begin{gather}\label{}
\pi_n(s_1)=\left(\frac{2\pi }{1-\omega_1}\right)^{1/2}\left(\frac{ne}{2}\right)^{n/2}n^{1/6}e^{t_1n^{1/3}}u_1(t_1;t_2-t_1)(1+O(n^{-1/3})),\label{pns1-asymptotic-expansion}\\
\pi_n(s_2)=\left(\frac{2\pi }{\omega_1-\omega_2}\right)^{1/2}\left(\frac{ne}{2}\right)^{n/2}n^{1/6}e^{t_2n^{1/3}}u_2(t_1;t_2-t_1)(1+O(n^{-1/3})).
\label{pns2-asymptotic-expansion}\end{gather}
Here  $v_k(x;s)$ and $u_k(x;s)$, $k=1,2$ are solutions to   \eqref{int-equation v} and \eqref{int-equation u}  subject to the boundary conditions \eqref{eq:u-asy}, $H_{II}(x;s)$ is the Hamiltonian correponding to these solutions.
\end{thm}

\begin{rem} \label{rem:OPOneJump}
When $s_1\to s_2$, the weight function \eqref{weight} is reduced to the Gaussian weight with one jump discontinuity.  Then, Theorem \ref{thm:OpAsy}, together with Remark \ref{rem:HankelAsy}, implies the asymptotics of the recurrence coefficients, leading coefficients and  the orthogonal polynomials associated with  Gaussian weight with one jump discontinuity. This agrees with a result from \cite[Theorem 5]{BC}.
\end{rem}
\subsection{Organization of the rest of this paper}

The rest of the paper is organized as follows.  In Section \ref{OPCPIV}, we consider the Riemann-Hilbert (RH) problem  for the orthogonal polynomials associated with the Gaussian weight with two jump discontinuities  \eqref{weight}.  We show that the RH problem is equivalent to the one  for the coupled Painlev\'e IV system. The properties of the Painlev\'e IV system are studied, including the Lax pair and the Hamiltonian formulation.   We then prove Theorem \ref{thm:HankelFixedn}  at the end of  this section which relates the Hankel determinants and the orthogonal polynomials to the coupled Painlev\'e IV system.  In section \ref{OPAsy}, we study the asymptotics of the orthogonal polynomials by performing Deift-Zhou steepest descent analysis
 of the RH problem for the  orthogonal polynomials. Finally, the proofs of Theorem \ref{thm:HankelAsy}-\ref{thm:OpAsy} are given in Section \ref{ProofsofTherorem}.

\section{Orthogonal polynomials and the coupled Painlev\'e IV system}\label{OPCPIV}
In this section, we will relate the the Hankel determinants and  the orthogonal polynomials associated with the weight function  \eqref{weight}  to the coupled Painlev\'e IV system. The cennections are collected in Theorem \ref{thm:HankelFixedn}. The derivations  are based on the RH problem representation of the orthogonal polynomials.
\subsection{Riemann-Hilbert problem for the orthogonal polynomials}
In this subsection,  we first consider the RH problem for the orthogonal polynomials with respect to \eqref{weight},  which was introduced by Fokas, Its and Kitaev \cite{Fokas}. We  then derive several identities relating the logarithmic derivative of the Hankel determinants to the
RH problem.  At the end of the subsection, we transform the RH problem to a model RH problem with constant jumps.

\subsection*{Riemann-Hilbert problem  for  $Y$ }

\begin{description}
  \item(a)~~  $Y(z;s_1,s_2)$ ($Y(z)$ for short) is analytic in
  $\mathbb{C}\backslash \mathbb{R}$;

  \item(b)~~  $Y(z)$  satisfies the jump condition
  $$Y_+(x)=Y_-(x) \left(
                               \begin{array}{cc}
                                 1 & w(x) \\
                                 0 & 1 \\
                                 \end{array}
                             \right),
\qquad x\in \mathbb{R},$$ where $w(x)=w(x;s_1,s_2;\omega_1,\omega_2)$ is defined in \eqref{weight};

  \item(c)~~  The behavior of $Y(z)$ at infinity is
  \begin{equation}\label{Y-infinity}Y(z)=\left (I+\frac {Y_{1}}{z}+O\left (\frac 1 {z^2}\right )\right )\left(
                               \begin{array}{cc}
                                 z^n & 0 \\
                                 0 & z^{-n} \\
                               \end{array}
                             \right),\quad \quad z\rightarrow
                             \infty ;\end{equation}
\item(d)~~  $Y(z)=O(\ln |z-s_k|) $ as $z\to s_k$ for $k=1,2$.
\end{description}

For $\omega_k\geqslant 0$, $k=1,2$,  it follows from the Sokhotski-Plemelj formula and Liouville's theorem that the unique solution of the RH problem for $Y$ is given by
\begin{equation}\label{Y}
Y(z)= \left (\begin{array}{cc}
\pi_n(z)& \frac 1{2\pi i}\int_\mathbb{R} \frac{\pi_n(x) w(x)}{x-z}dx\\
-2\pi i \gamma_{n-1}^2 \;\pi_{n-1}(z)& - \gamma_{n-1}^2\;
\int_\mathbb{R} \frac{\pi_{n-1}(x) w(x)}{x-z}dx\end{array} \right ),
\end{equation}
where  $\pi_n(z)$ and  $\gamma_{n-1}$ are defined in \eqref{def: Orth}; see  \cite{Fokas}.

We establish two differential identities  expressing the  logarithmic derivative  of the  Hankel determniant $D_n$ in terms of the solution  $Y$.
\begin{pro}\label{Pro:Diff} Let $s_k$ and $\omega_k$, $k=1,2$ be as in \eqref{def:Para} and $F(s_1, s_2)$ be  the  logarithmic derivative  of the  Hankel determinant  as defined in \eqref{def:F},
 we have the following  relations
\begin{equation}\label{def: DiffId-1}
 F(s_1,s_2)=\frac{1-\omega_1}{2\pi i}e^{-s_1^2}(Y^{-1}Y')_{21}(s_1)+\frac{\omega_1 -\omega_2}{2\pi i}e^{-s_2^2}(Y^{-1}Y')_{21}(s_2),
  \end{equation}
  and
  \begin{equation}\label{def: DiffId-2}
 F(s_1,s_2)= 2\lim_{z\to\infty}z(Y(z)z^{-n\sigma_3}-I)_{11},
  \end{equation}
  where $Y$ is defined in \eqref{Y}.
    \end{pro}
\begin{proof}

According to \eqref{def:F}, it follows by taking logarithmic derivative on both sides of the equation \eqref{eq:DGamma}  that
\begin{align}\label{dF-0}
F(s_1,s_2)&=-2\sum_{j=0}^{n-1}\gamma_j^{-1}\left(\frac{\partial\gamma_j }{\partial s_1}+\frac{\partial\gamma_j }{\partial s_2}\right)\\ \label{dF-1}
& =\sum_{j=0}^{n-1}((1-\omega_1)e^{-s_1^2}\gamma_j^{2}\pi_j(s_1)^2+(\omega_1-\omega_2)e^{-s_2^2}\gamma_j^{2}\pi_j(s_2)^2).
  \end{align}
Applying the Christoffel-Darboux identity, we obtain
\begin{align}\label{dF-2}
 F(s_1,s_2)=& (1-\omega_1)e^{-s_1^2}\gamma_{n-1}^{2}(\pi_{n}'(s_1)\pi_{n-1}(s_1) -\pi_{n}(s_1)\pi'_{n-1}(s_1))\nonumber\\
&~  +(\omega_1-\omega_2)e^{-s_2^2}\gamma_{n-1}^{2}(\pi_{n}'(s_2)\pi_{n-1}(s_2)-\pi_{n}(s_2)\pi'_{n-1}(s_2)).
  \end{align}
 Then, the differential identity \eqref{def: DiffId-1} follows from the definition of $Y$ and \eqref{dF-2}.

 To prove \eqref{def: DiffId-2}, we use a change of variable in \eqref{def: Orth} and obtain
\begin{equation}\label{leading coef1}
\gamma_{j}^{-2}=\gamma_{j}(s_1,s_2)^{-2}=\int_{\mathbb{R}}\pi_j(x)^2w(x)dx=\int_{\mathbb{R}}\pi_j(x+s_k)^2w(x+s_k)dx, \quad k=1,2.
\end{equation}
Taking derivative with respect to $s_k$   on both sides of \eqref{leading coef1} for $k=1,2$ and using the orthogonality and the definition of the weight function \eqref{weight}, we have
\begin{align}\label{leading coef2}-2\gamma_{j}^{-1}\frac \partial {\partial s_1}\gamma_{j}
&=-2\gamma_{j}^2\int_{\mathbb{R}}x\pi_j(x)^2w(x)dx-
(\omega_1-\omega_2)e^{-s_2^2}\gamma_j^{2}\pi_j(s_2)^2,
\end{align}
and
\begin{align}\label{leading coef3}-2\gamma_{j}^{-1}\frac \partial {\partial s_2}\gamma_{j}
&=-2\gamma_{j}^2\int_{\mathbb{R}}x\pi_j(x)^2w(x)dx-
(1-\omega_1)e^{-s_1^2}\gamma_j^{2}\pi_j(s_1)^2.
\end{align}
Combining  the formulas with   \eqref{dF-0}-\eqref{dF-1} and using the Christoffel-Darboux formula once again, we obtain
\begin{equation}\label{eq:F-Pi}F(s_1,s_2)=-2\gamma_{n-1}^2\int_{\mathbb{R}}x\left(\frac d {dx}\pi_n(x) \pi_{n-1}(x)-\pi_n(x) \frac d {dx}\pi_{n-1}(x)\right)w(x)dx.
\end{equation}
From
$$\pi_n(x)=x^n+p_nx^{n-1}+\cdots, $$
we have the decomposition
$$x\frac d {dx} \pi_n(x)=n\pi_n(x)-p_n\pi_{n-1}(x)+\cdots.$$
Substituting this into \eqref{eq:F-Pi} and using the orthogonality, we obtain \eqref{def: DiffId-2}. This completes Proposition \ref{Pro:Diff}.
\end{proof}

 We define
 \begin{equation}\label{eq: Phi} \Phi(z;x,s)=\sigma_1e^{\frac{x^2}{2}\sigma_3} Y(z+x)e^{- \frac{1}{2}(z+x )^2\sigma_3}\sigma_1,
  \end{equation}
  where the variables $x,s$ are related to $s_1$ and $s_2$ by \eqref{def:x}.
Then $\Phi(z)=\Phi(z;x,s)$ satisfies the following RH problem.

\subsection*{Riemann-Hilbert problem for $\Phi$}
\begin{description}
  \item(a)~~  $\Phi(z)$ is analytic in
  $\mathbb{C}\backslash \mathbb{R}$;

  \item(b)~~  $\Phi(z)$  satisfies the jump condition
  $$\Phi_+(z)=\Phi_-(z) \left(
                               \begin{array}{cc}
                                1 & 0 \\
                                 1 & 1 \\
                                 \end{array}
                             \right), \quad z<-s, $$
                             $$\Phi_+(z)=\Phi_-(z) \left(
                               \begin{array}{cc}
                                 1  & 0 \\
                                \omega_1 &1 \\
                                 \end{array}
                             \right), \quad -s<z<s;
$$
$$ \Phi_+(z)=\Phi_-(z) \left(
                               \begin{array}{cc}
                                1 & 0 \\
                               \omega_2 &  1 \\
                                 \end{array}
                             \right), \quad z>s;
$$

  \item(c)~~  The behavior of $\Phi(z)$ at infinity is
  \begin{equation}\label{Phi-infinity}\Phi(z)=\left (I+\frac {\Phi_{1}}{z}+\frac {\Phi_{2}}{z^2}+O\left (\frac 1 {z^3}\right )\right )e^{(\frac{1}{2}z^2+x z)\sigma_3}z^{-n\sigma_3};\end{equation}

\item(d)~~ The behavior of $ \Phi(z)$ near $-s$ is
\begin{equation}\label{Phi--s}
\Phi(z)=\Phi^{(-s)}(z)\left (I+ \frac{1-\omega_1}{2\pi i}\left(\begin{array}{cc} 0 & 0\\ 1 & 0\end{array}\right)\ln (z+s)\right )E^{(-s)},
\end{equation}
where $\arg(z+s)\in(-\pi,\pi)$. Here,   $\Phi^{(-s)}(z)$ is analytic near $z=-s$  and has the following expansion

\begin{equation}\label{def:P-0}
\Phi^{(-s)}(z)=P_0(x, s)(I+P_1(x, s)(z+s)+O((z+s)^2)).\end{equation}
The piecewise constant matrix $E^{(-s)}$ is given by
 \begin{equation*}
 E^{(-s)}=\left\{\begin{array}{ll}
 \left(\begin{array}{cc} 1 & 0\\ 0 & 1\end{array}\right), & \Im z>0,\\
 \left(\begin{array}{cc} 1 & 0\\ -\omega_1 & 1\end{array}\right),  & \Im z<0.
 \end{array}\right.
 \end{equation*}
\item(e)~~ The behavior of $ \Phi(z)$ near $s$ is
\begin{equation}\label{Phi-s}
\Phi(z)=\Phi^{(s)}(z)\left (I+ \frac{\omega_1-\omega_2}{2\pi i}\left(\begin{array}{cc} 0 & 0\\ 1 & 0\end{array}\right)\ln (z-s)\right )E^{(s)},
\end{equation}
where $\arg(z-s)\in(-\pi,\pi)$. Here,  $\Phi^{(s)}(z)$ is analytic near $z=s$  and has the following expansion
\begin{equation}\label{def:Q-0}
\Phi^{(s)}(z)=Q_0(x,s)(I+Q_1(x,s)(z-s)+O((z-s)^2)).\end{equation}
The piecewise  constant matrix $E^{(s)}$ is defined by
 \begin{equation*}
 E^{(s)}=\left\{\begin{array}{ll}
 \left(\begin{array}{cc} 1 & 0\\ 0 & 1\end{array}\right), & \Im z>0,\\
 \left(\begin{array}{cc} 1 & 0\\ -\omega_2 & 1\end{array}\right),  & \Im z<0.
 \end{array}\right.
 \end{equation*}

\end{description}


\subsection{Lax pair and the coupled Painlev\'e IV system}
In this section, we show that the solution $\Phi(z;x,s)$  of the RH problem  satisfies a system of
differential equations in $z$ and $x$ when the parameter $s$ is fixed. The compatibility condition $\Phi_{zx}(z;x,s)=\Phi_{xz}(z;x,s)$ gives us the
coupled Painlev\'e IV system. The Hamiltonian for the system is also derived.

\begin{pro} \label{pro:Laxpair} We have the following  Lax pair
\begin{equation}\label{def: Lax pair}
 \Phi_z(z;x,s)=A(z;x,s) \Phi(z;x,s),  \quad   \Phi_x(z;x,s)=B(z;x,s) \Phi(z;x,s),
  \end{equation}
  where
  \begin{equation}\label{eq: A}
 A(z;x,s)=(z+x)\sigma_3+A_{\infty}(x,s)+\frac{A_1(x,s)}{z+s}+\frac{A_2(x,s)}{z-s},  \end{equation}
  \begin{equation}\label{eq: B}
 B(z;x,s)=z\sigma_3+A_{\infty}(x,s), \end{equation}
with the coefficients  given below
   \begin{equation}\label{eq:A-infty}
 A_{\infty}(x,s)= \begin{pmatrix}
                    0& y(x;s)\\
                 -2 \left(a_1(x;s)b_1(x;s)+a_2(x;s)b_2(x;s)+n\right)/y(x;s) &  0
                  \end{pmatrix}, \end{equation}
  \begin{equation}\label{eq:A-k}
 A_k(x,s)= \begin{pmatrix}
                   a_k(x;s)b_k(x;s) &  a_k(x;s)y(x;s)\\
               -a_k(x;s)b_k^2(x;s)/y(x;s)&  -a_k(x;s)b_k(x;s)              \end{pmatrix},  \quad k=1,2. \end{equation}
The compatibility condition of the Lax pair gives us the coupled Painlev\'e IV system  \begin{equation}\label{eq:CPIV}
\left\{\begin{array}{l}
\frac{d y}{d x}=2(a_1+a_2-x)y,\\
\frac{d a_1}{d x}=-2a_1(a_1+a_2+b_1-x+s),\\
\frac{d a_2}{d x}=-2a_2(a_1+a_2+b_2-x-s),\\
\frac{d b_1}{d x}=b_1^2+2b_1(2a_1+a_2-x+s)+2(a_2b_2+n),\\
\frac{d b_2}{d x}=b_2^2+2b_2(a_1+2a_2-x-s)+2(a_1b_1+n).
\end{array}\right.
\end{equation}
Eliminating $b_1$ and $b_2$ from the system, it is seen that $a_1$ and $a_2$ satisfy the following nonlinear differential equations
\begin{equation}\label{eq:ak}
\left\{\begin{array}{l}
       \frac{\mathtt{d} ^2a_1 }{\mathtt{d} x^2}-\frac 1{2a_1} \left(\frac{\mathtt{d} a_1}{\mathtt{d} x}\right)^2-6a_1(a_1+a_2)^2+8a_1(a_1+a_2)x-8a_1^2s+2(2n-1)a_1-2a_1(x-s)^2=0,\\
         \frac{\mathtt{d}^2a_2}{\mathtt{d} x^2}-\frac 1{2a_2} \left(\frac{\mathtt{d} a_2}{\mathtt{d} x}\right)^2-6a_2(a_1+a_2)^2+8a_2(a_1+a_2)x+8a_2^2s+2(2n-1)a_2-2a_2(x+s)^2=0.
       \end{array}
\right.
\end{equation}

    \end{pro}
 \begin{proof}
 Since the  jump matrices of the RH problem for $\Phi(z;x,s)$ are independent of the variables $z$ and $x$,
 we have that $\Phi_z(z;x,s)$, $\Phi_x(z;x,s)$ and $\Phi(z;x,s)$ satisfy the same jump condition.
 Thus, the coefficient $A(z;x,s)$ in the differential equations
 are meromorphic for $z$ in the complex plane with only possible isolate singularities  at $z=0$, $\pm s$ and the coefficient $B(z;x,s)$ is analytic for $z$ in the complex plane.
 Then, it follows from the local behavior of $\Phi(z;x,s)$ as $z\to\infty$, $z\to \pm s$ that
 the coefficients $A(z;x,s)$ and $B(z;x,s)$ are rational functions in $z$ with the form given in \eqref{eq: A}-\eqref{eq: B}.
 Using the fact that $\det \Phi=1$, we have $\tr A=\tr B=0$ and thus all the coefficients $A_k$, $k=0,1,2$ in  \eqref{eq: A} are trace-zero. Using the master equation in \eqref{def: Lax pair} and  the local behavior $\Phi(z)$ at $z \pm s$, we have
 \begin{equation*}
  \det A_k=0, \quad k=1,2.
 \end{equation*}
 We denote  $(A_k)_{11}=a_kb_k$ for  $k=1,2$.

Substituting the behavior of $\Phi $ at infinity into the master equation of the Lax pair \eqref{def: Lax pair}, we find after comparing the coefficients of  $z^0$ and $z^{-1}$  on both sides of the equation that
  \begin{equation}\label{eq:z-0}A_{\infty}=[\Phi_1, \sigma_3],
              \end{equation}
              and
\begin{equation}\label{eq:A-k-Psi}A_1+A_2=-n\sigma_3 +[\Phi_2+x\Phi_1, \sigma_3]+[\sigma_3,\Phi_1]\Phi_1,            \end{equation}
              where $\Phi_k$ is the coefficient of $z^{-k}$ in the large $z$ asymptotic expansion of $\Phi(z)$.
In view of  \eqref{eq:z-0}, we  get
 \begin{equation}\label{eq:A-infty-1}
 A_{\infty}= \begin{pmatrix}
                    0& -2(\Phi_1)_{12}\\
                 2(\Phi_1)_{21} &  0
                  \end{pmatrix}. \end{equation}
From the diagonal entries of the equation \eqref{eq:A-k-Psi}, we find the relation
\begin{equation}\label{eq:Phi1ab1}
2(\Phi_1)_{12}(\Phi_1)_{21} = n+(A_1+A_2)_{11}=a_1b_1+a_2b_2+n.\end{equation}
 We define
 \begin{equation}\label{def:y}
 y=-2(\Phi_1)_{12},
 \end{equation}
 then the above relations imply that
 \begin{equation}\label{eq:A-infty-21}
(A_{\infty})_{12}=y, \quad (A_{\infty})_{21}=-\frac{2}{y}(a_1b_1+a_2b_2+n).\end{equation}
We define  $(A_k)_{12}=a_ky$ for  $k=1,2$.
 Then, the other entries of $A_k$ can be expressed in terms of $a_k$ and $b_k$ for $k=1,2$, as given in \eqref{eq:A-k}.

 Similarly, the coefficient $B(z)=B(z;x,s)$ can be determined by using the behavior of $\Phi $ at infinity
\begin{equation}\label{eq:B-Phi}B(z)=\Phi_x(z)\Phi(z)^{-1}=z\sigma_3+\begin{pmatrix}
                0&  -2(\Phi_1)_{12}\\
              2(\Phi_1)_{21} & 0     \end{pmatrix}=z\sigma_3+A_{\infty}.
              \end{equation}

 The compatibility condition  $\Phi_{zx}=\Phi_{xz}$ gives  us the zero-curve equation
 \begin{equation}\label{eq:ZeroCurve}A_x-B_z+[A,B]=0.\end{equation}
 Substituting \eqref{eq: A} and \eqref{eq: B} into the above equation, the compatibility condition  is equivalent to 
\begin{equation}\label{eq:dA}
\left\{\begin{array}{l}
        \frac{dA_{\infty}}{dx}=x[A_{\infty},\sigma_3]-[A_1,\sigma_3]-[A_2,\sigma_3], \\
         \frac {dA_1}{dx}=-[A_1, A_{\infty}]+s[A_1,\sigma_3], \\
          \frac {dA_2}{dx}=-[A_2, A_{\infty}]-s[A_2,\sigma_3].
       \end{array}
\right.
\end{equation}
  We then obtain the system of differential equations \eqref{eq:CPIV}.
  Deleting $b_1$ and $b_2$ from the system, we obtain the differential equations for $a_1$ and $a_2$ as given in \eqref{eq:ak}.
  This completes the proof of Proposition \ref{pro:Laxpair}.
%

 \end{proof}

\begin{pro}\label{pro:Ham}
The Hamiltonian for the coupled Painlev\'e IV system is
\begin{equation}\label{eq: H}
H_{\texttt{IV}}(a_1,a_2,b_1,b_2; x,s)=-2(a_1b_1+a_2b_2+n)(a_1+a_2)+2(a_1b_1(x-s)+a_2b_2(x+s)+nx)-(a_1b_1^2+a_2b_2^2).
  \end{equation}
And the coupled Painlev\'e IV system  \eqref{eq:CPIV} can be written as  the Hamiltonian system:
  \begin{equation}\label{eq: H-system}
   \left\{\begin{array}{l}
\frac{d a_1}{d x}=H_{\texttt{IV}, b_1}(a_1,a_2,b_1,b_2; x,s),\\
\frac{d a_2}{d x}=H_{\texttt{IV}, b_2}(a_1,a_2,b_1,b_2; x,s),\\
\frac{d b_1}{d x}=-H_{\texttt{IV}, a_1}(a_1,a_2,b_1,b_2; x,s),\\
\frac{d b_2}{d x}=-H_{\texttt{IV}, a_2}(a_1,a_2,b_1,b_2; x,s),
\end{array}\right.
\end{equation}
where  $H_{\texttt{IV}, a}$ denotes the partial derivative of $H_{\texttt{IV}}$ with respect to $a$.
     \end{pro}

 \begin{proof}
The Hamiltonian introduced by Jimbo, Miwa and Ueno \cite{jmu} is given by
 \begin{equation}\label{def:H}
 H_{\texttt{IV}}(x;s)=-\mathtt{Res}_{z=\infty} \Phi^{(\infty)}(z)^{-1}\frac{d}{dz}\Phi^{(\infty)}(z)\frac{d}{dx}\Theta(z;x)\sigma_3=-2(\Phi_1)_{11},\end{equation}
 where $\Theta(z;x)=(\frac{1}{2}z^2+x z)\sigma_3$ and $\Phi_1$ is the coefficient of $z^{-1}$ in the large $z$ expansion of $\Phi $ in \eqref{Phi-infinity}.
Using \eqref{def: DiffId-1}, \eqref{def: DiffId-2} and \eqref{eq: Phi}, we have
\begin{align}\label{H-F}
H_{\texttt{IV}}(x;s)&=F(s_1,s_2)+2nx\\
&=2nx+\frac{1-\omega_1}{2\pi i}
    \left(\Phi^{-1}\Phi_z
    \right)_{12}(-s)
    +\frac{\omega_1-\omega_2}{2\pi i}
    \left(\Phi^{-1}\Phi_z
    \right)_{12}(s). \nonumber
\end{align}
From \eqref{Phi--s} and \eqref{Phi-s}, we get
\begin{equation}\label{eq:H-P}
H_{\texttt{IV}}(x;s)=2nx+\frac{1-\omega_1}{2\pi i}
    \left(P_1
    \right)_{12}
    +\frac{\omega_1-\omega_2}{2\pi i}
    \left(Q_1
    \right)_{12},
\end{equation}
where $P_1=P_1(x,s)$ and $Q_1=Q_1(x,s)$ are defined in \eqref{def:P-0} and \eqref{def:Q-0}, respectively.
Substituting the expansions \eqref{Phi--s} and \eqref{Phi-s} into the master equation of \eqref{def: Lax pair}, we obtain
\begin{gather}
\frac{1-\omega_1}{2\pi i}P_0\left(\begin{array}{cc} 0 & 0\\ 1 & 0\end{array}\right)P_0^{-1}=A_1
,\label{P0}\\
 P_1+\frac{1-\omega_1}{2\pi i}\left[P_1,\left(\begin{array}{cc} 0 & 0\\ 1 & 0\end{array}\right)\right]=P_0^{-1}((x-s)\sigma_3+A_{\infty}-\frac{A_2}{2s})P_0, \label{P1}\\
 \frac{\omega_1-\omega_2}{2\pi i}Q_0\left(\begin{array}{cc} 0 & 0\\ 1 & 0\end{array}\right)Q_0^{-1}=A_2
,\label{Q0}\\
 Q_1+\frac{\omega_1-\omega_2}{2\pi i}\left[Q_1,\left(\begin{array}{cc} 0 & 0\\ 1 & 0\end{array}\right)\right]=Q_0^{-1}(
 (x+s)\sigma_3+A_{\infty}+\frac{A_1}{2s})Q_0. \label{Q1}
\end{gather}
Now
$$P_0\left(\begin{array}{cc} 0 & 0\\ 1 & 0\end{array}\right)P_0^{-1}= \left(
                               \begin{array}{cc}
                                  (P_0)_{12}(P_0)_{22}& -(P_0)_{12}^2 \\
                                 (P_0)_{22}^2 & - (P_0)_{12}(P_0)_{22} \\
                                 \end{array}
                             \right).$$
Then,  a substitution of the above equation into \eqref{P0} gives
\begin{equation}\label{eq:pij}
\frac{1-\omega_1}{2\pi i}(P_0)_{22}^2=-\frac{a_1b_1^2}{y}, \quad \frac{1-\omega_1}{2\pi i}(P_0)_{12}^2=-a_1y, \quad
\frac{1-\omega_1}{2\pi i}(P_0)_{12}(P_0)_{22}=a_1b_1.
\end{equation}
Let
$$\mathcal{A}=(x-s)\sigma_3+A_{\infty}-\frac{A_2}{2s},$$
we obtain from \eqref{P1} that
\begin{align}\label{eq:H-P-1}
\frac{1-\omega_1}{2\pi i}(P_1)_{12}&=\frac{1-\omega_1}{2\pi i}(2(P_0)_{12}(P_0)_{22}\mathcal{A}_{11}-(P_0)_{12}^2\mathcal{A}_{21}+(P_0)_{22}^2\mathcal{A}_{12})\nonumber\\
&=2a_1b_1\mathcal{A}_{11}+a_1y \mathcal{A}_{21}-\frac{a_1b_1^2}{y}\mathcal{A}_{12}\nonumber\\
    &=2a_1b_1(x-s)-2a_1(a_1b_1+a_2b_2+n)-a_1b_1^2+\frac{1}{2s}a_1a_2(b_1-b_2)^2.
\end{align}
Similarly, we get after some straightforward calculations
\begin{equation}\label{eq:qij}
\frac{\omega_1-\omega_2}{2\pi i}(Q_0)_{22}^2=-\frac{a_2b_2^2}{y}, \quad \frac{\omega_1-\omega_2}{2\pi i}(Q_0)_{12}^2=-a_2y, \quad
\frac{\omega_1-\omega_2}{2\pi i}(Q_0)_{12}(Q_0)_{22}=a_2b_2,
\end{equation}
and
\begin{equation}\label{eq:H-P-2}
\frac{\omega_1-\omega_2}{2\pi i}(Q_1)_{12}=2a_2b_2(x+s)-2a_1(a_1b_1+a_2b_2+n)-a_2b_2^2-\frac{1}{2s}a_1a_2(b_1-b_2)^2.
\end{equation}
Then, the expression of the  Hamiltonian  \eqref{eq: H} follows directly by substituting  \eqref{eq:H-P-1}  and \eqref{eq:H-P-2} into  \eqref{eq:H-P}.   In view of  the Hamiltonian  \eqref{eq: H},   it is seen that the coupled Painlev\'e IV system  \eqref{eq:CPIV} is equivalent to the Hamiltonian system \eqref{eq: H-system}. This completes the proof of Proposition \ref{pro:Ham}.
\end{proof}

\subsection{Proof of Theorem \ref{thm:HankelFixedn}}
The relation \eqref{thm: F-H} follows  from \eqref{H-F}. 
Let $Y_1$ and  $Y_2$ be the   coefficients of $1/z$ and $1/z^2$ in the expansion of $Y$ near infinity  \eqref{Y-infinity},
we have the following relations for the recurrence coefficients  $\alpha_n=\alpha_n(s_1,s_2)$,  $\beta_{n-1}=\beta_{n-1}(s_1,s_2)$ and  the leading coefficient  $\gamma_{n}=\gamma_{n}(s_1,s_2)$ of the monic orthogonal polynomial of degree $n-1$:
\begin{equation}\label{recurrence coefficients and Y}
\alpha_n=(Y_1)_{11}+\frac {(Y_2)_{12}}{(Y_1)_{12}},\quad\quad
\beta^2_n=(Y_1)_{12}(Y_1)_{21} \quad\mbox{and}\quad \gamma_{n-1}^2=-\frac{1}{2\pi i}(Y_1)_{21};
\end{equation}
see \cite{DeiftZhouS}.
In view of \eqref{eq: Phi}, it is then  seen that
\begin{equation}\label{eq: Y1Phi1}
(Y_1)_{11}=-(\Phi_1)_{11}-nx, \quad \quad  (Y_1)_{12}=e^{-x^2}(\Phi_1)_{21}, \quad \quad  (Y_1)_{21}=e^{x^2}(\Phi_1)_{12},
\end{equation}
and
\begin{equation}\label{eq: Y2Phi2}
(Y_2)_{12}=e^{-x^2}((n+1)x(\Phi_1)_{21}+(\Phi_2)_{21} ),\end{equation}
where $\Phi_1$ and $\Phi_2$ are defined in \eqref{Phi-infinity}.
From the relation \eqref{eq:A-k-Psi}, we have
\begin{equation}\label{eq:Phi2-A}
x(\Phi_1)_{21}+(\Phi_2)_{21}=\frac {1}{2}(A_1+A_2)_{21}+(\Phi_1)_{11}(\Phi_1)_{21}.
\end{equation}
Substituting \eqref{eq: Y1Phi1}  into \eqref{recurrence coefficients and Y} and recalling \eqref{eq:Phi1ab1}-\eqref{eq:A-infty-21}, we  obtain that
\begin{equation}\label{eq: beta-ab}
 \beta^2_n=(\Phi_1)_{12}(\Phi_1)_{21} =\frac {1}{2}\left(a_1(x;s)b_1(x;s)+a_2(x;s)b_2(x;s)+n\right),
\end{equation}
and
\begin{equation}\label{eq:leading-ab}
\gamma_{n-1}^2=-\frac{1}{2\pi i}e^{x^2}(\Phi_1)_{12}=\frac{1}{4\pi i}e^{x^2}y(x;s)\neq 0.
\end{equation}
On account of \eqref{eq:CPIV}, we have
\begin{equation}\label{eq:log-d-gamma}
\frac{d }{dx}\ln \gamma_{n-1}=a_1(x;s)+a_2(x;s).
\end{equation}
Inserting \eqref{eq: Y1Phi1}-\eqref{eq:Phi2-A} into \eqref{recurrence coefficients and Y} yields
\begin{align}\label{eq:alpha-ab}
\alpha_n&=-(\Phi_1)_{11}+\frac {(\Phi_2)_{21}+x(\Phi_1)_{21}}{(\Phi_1)_{21}}\nonumber\\
&=\frac {1}{2}\frac{(A_1+A_2)_{21}}{(\Phi_1)_{21}}\nonumber\\
&=\frac{a_1(x;s)b_1^2(x;s)+a_2(x;s)b_2^2(x;s)}{2(a_1(x;s)b_1(x;s)+a_2(x;s)b_2(x;s)+n)}.
\end{align}
In summary,  we obtain \eqref{thm:alpha}-\eqref{thm: gamma} by collecting   \eqref{eq: beta-ab}-\eqref{eq:alpha-ab}.

From \eqref{Y}, \eqref{eq: Phi} and \eqref{eq:pij}, it is seen that
\begin{equation}\label{eq:pns1}
\pi_n(s_1)=(Y)_{11}(s_1)=e^{\frac{1}{2}s^2-sx}(\Phi)_{22}(-s)=e^{\frac{1}{2}s^2-sx}(P_0)_{22},
\end{equation}
and
\begin{equation}\label{eq:pns2}
\pi_n(s_2)=(Y)_{11}(s_2)=e^{\frac{1}{2}s^2+sx}(\Phi)_{22}(s)=e^{\frac{1}{2}s^2+sx}(Q_0)_{22}.
\end{equation}
Therefore, we obtain \eqref{thm:pns1}  and \eqref{thm:pns2} after replacing  the expressions of $(P_0)_{22}$  and $(Q_0)_{22}$ by \eqref{eq:pij} and \eqref{eq:qij}.  This completes the proof of Theorem \ref{thm:HankelFixedn}.

\section{Nonlinear steepest descent analysis of the Riemann-Hilbert  problem for $Y$}\label{OPAsy}
In this section, we take $s_1=\sqrt{2n}+\frac {t_1}{\sqrt{2}n^{1/6}}$  and $s_2=\sqrt{2n}+\frac {t_2}{\sqrt{2}n^{1/6}}$ in the weight function \eqref{weight}. Then, we perform Deift-Zhou nonlinear steepest descent analysis \cite{DeiftZhouU,DeiftZhouS,DeiftZhouA} for the Riemann-Hilbert  problem for $Y(z;s_1,s_2)$ as $n\to \infty$.
The analysis will allow us to find the asymptotics of Hankel determinants and the orthogonal polynomials associated with \eqref{weight}.
The analysis of a  Riemann-Hilbert  problem with one jump singularity in the weight function \eqref{weight} is considered in \cite{XuZhao} by the second author and Zhao.
\subsection{The  first transformation: $Y\rightarrow T$}

The first transformation is defined by
 \begin{equation}\label{Y to T}T(z)= (2n)^{-\frac 12n\sigma_3}e^{-\frac 1 2 nl \sigma_3} Y(\sqrt{2n}z) e^ {n \left (\frac 1 2 l
-g(z)\right )\sigma_3},\quad z\in
\mathbb{C}\backslash \mathbb{R}, \end{equation}
where the constant $l=-1-2\ln 2$. The  $g$-function therein is defined by
\begin{equation}\label{g}g(z)=\frac{2}{\pi}\int_{-1}^{1}  \ln(z-x) \sqrt{1-x^2} dx,  \end{equation}
where the logarithm takes the principle branch $\arg(z-x)\in (-\pi, \pi)$.
We then introduce the $\phi$-function
\begin{equation}\label{phi}
\phi(z)=z\sqrt{z^2-1}-\ln\left(z+\sqrt{z^2-1}\right), \end{equation}
where the principle branches are chosen.
The $\phi$-function and $g$-function are related by
\begin{equation}\label{phase condition}
2\left[g(z)+\phi(z)\right]-2z^2-l=0, \quad z\in \mathbb{C}\backslash(-\infty,1].
\end{equation}
As a consequence,   $T$ is normalized at infinity
$$T(z)=I+O(1/z),$$
and satisfies the jump
condition
\begin{equation}\label{Jump-T}T_+(x)=T_-(x)\left\{
\begin{array}{ll}
 \left(
                               \begin{array}{cc}
                                 1 & \theta(x)e^{-2n \phi(x)} \\
                                 0 & 1 \\
                               \end{array}
                             \right),&   x\in(1,+\infty);\\
                             &\\
  \left(
                               \begin{array}{cc}
  e^{2n\phi_+(x)} & \theta(x) \\                                  0 & e^{2n\phi_-(x)} \\
                               \end{array}
                             \right),&   x\in(-1, 1); \\
                             &\\
   \left(
                               \begin{array}{cc}
                                 1 &  e^{-2n \phi_+(x)} \\
                                 0 & 1 \\
                               \end{array}
                             \right), &  x\in(-\infty,-1),
\end{array}
\right .
\end{equation}
where $\theta(x)=\left\{\begin{array}{cc}
                                                          1& x< \lambda_1\\
                                                           \omega_1& \lambda_1<x<\lambda_2,\\
                                                        \omega_2& x>\lambda_2,
                                                       \end{array}
                                                       \right.$
with $\lambda_1=1+\frac {t_1}{2n^{2/3}}$ and $\lambda_2=1+\frac {t_2}{2n^{2/3}}$.

\subsection{The  second transformation: $T\rightarrow S$}

In the second  transformation, we define

 \begin{equation}\label{T-S}
S(z)=\left \{
\begin{array}{ll}
  T(z), & \mbox{for $z$ outside the lens,}
  \\ &\\
  T(z) \left( \begin{array}{cc}
                                 1 & 0 \\
                                   -e^{2n\phi(z)} & 1 \\
                               \end{array}
                             \right) , & \mbox{for $z$ in the upper lens,}\\ &\\
T(z) \left( \begin{array}{cc}
                                 1 & 0 \\
                                   e^{2n\phi(z)} & 1 \\
                               \end{array}
                             \right) , & \mbox{for $z$ in the lower lens,}
\end{array}\right .\end{equation}
where the regions are illustrated in Fig.\ref{fig:S}.
Then $S$ satisfies the jump condition
\begin{equation}\label{def:SJump}
S_+(z)=S_-(z)J_S(z).
\end{equation}

\begin{figure}[h]
\begin{center}
  \includegraphics[scale=0.9,bb=140 617 457 668]{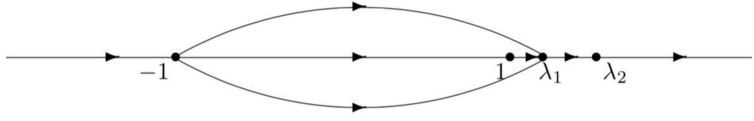}
 \caption{The jump contours and regions for the RH problem for $S$ when $\lambda_1>1$.}
\label{fig:S}
\end{center}
\end{figure}

%
%
For $\lambda_1>1$,  we have
\begin{equation} \label{def:SJump-1}
J_S(z)=\left\{
\begin{array}{ll}
\left(
                               \begin{array}{cc}
                                 1 & \omega_2e^{-2n\phi(z)} \\
                                 0 & 1 \\
                               \end{array}
                             \right),&   z\in(\lambda_2,+\infty),\\[.4cm]
\left(
                               \begin{array}{cc}
                                 1 & \omega_1e^{-2n\phi(z)} \\
                                 0 & 1 \\
                               \end{array}
                             \right),&   z\in(\lambda_1,\lambda_2),\\[.4cm]
  \left(
                               \begin{array}{cc}
                                 0 & e^{-2n\phi(z)}\\
                                 -e^{2n\phi(z)} & 0 \\
                               \end{array}
                             \right),&   z\in(1,\lambda_1),\\[.4cm]
   \left(
                               \begin{array}{cc}
                                 0 & 1\\
                                 - 1 & 0 \\
                               \end{array}
                             \right),&   z\in(-1, 1)\\
                             \left( \begin{array}{cc}
                                 1 & 0 \\
                                   e^{2n\phi(z)} & 1 \\
                               \end{array}
                             \right),&   z~\mbox{on lens}, \\[.4cm]

   \left(
                               \begin{array}{cc}
                                 1 &  e^{-2n \phi_+(z)} \\
                                 0 & 1 \\
                               \end{array}
                             \right), &  z\in(-\infty,-1),
\end{array}\right.
\end{equation}
where the contours are indicated in Fig. 1.\\
For $\lambda_1<1<\lambda_2$, we have
\begin{equation*}J_S(z)=\left\{
\begin{array}{ll}
 \left(
                               \begin{array}{cc}
                                 1 & \omega_2e^{-2n\phi(z)} \\
                                 0 & 1 \\
                               \end{array}
                             \right),&   z\in(\lambda_2,+\infty),\\[.4cm]
\left(
  \begin{array}{cc}
                                 1 & \omega_1e^{-2n\phi(z)} \\
                                 0 & 1 \\
                               \end{array}
                             \right),&   z\in(1,\lambda_2),\\[.4cm]
 \left(
 \begin{array}{cc}
                                e^{2n\phi_+(z)} & \omega_1 \\
                                 0 & e^{2n\phi_-(z)}\\
                               \end{array}
                             \right),&   z\in(\lambda_1,1),\\[.4cm]
                               \left(
                               \begin{array}{cc}
                                 0 & 1 \\
                                 - 1 & 0 \\
                               \end{array}
                             \right),&   z\in(-1, \lambda_1 ), \\[.4cm]

  \left( \begin{array}{cc}
                                 1 & 0 \\
                                   e^{2n\phi(z)} & 1 \\
                               \end{array}
                             \right),&   z~\mbox{on lens}, \\[.4cm]

   \left(
                               \begin{array}{cc}
                                 1 &  e^{-2n \phi_+(z)} \\
                                 0 & 1 \\
                               \end{array}
                             \right), &  z\in(-\infty,-1).
\end{array}
\right.
\end{equation*}

For $\lambda_1<\lambda_2<1$, we have
\begin{equation*}J_S(z)=\left\{
\begin{array}{ll}
 \left(
                               \begin{array}{cc}
                                 1 & \omega_2e^{-2n\phi(z)} \\
                                 0 & 1 \\
                               \end{array}
                             \right),&   z\in(1,+\infty),\\[.4cm]
\left(
 \begin{array}{cc}
                                e^{2n\phi_+(z)} & \omega_2 \\
                                 0 & e^{2n\phi_-(z)}\\
                               \end{array}
                             \right),&   z\in(\lambda_2,1),\\[.4cm]
 \left(
 \begin{array}{cc}
                                e^{2n\phi_+(z)} & \omega_1 \\
                                 0 & e^{2n\phi_-(z)}\\
                               \end{array}
                             \right),&   z\in(\lambda_1,\lambda_2),\\[.4cm]
  \left(
                               \begin{array}{cc}
                                 0 & 1 \\
                                 - 1 & 0 \\
                               \end{array}
                             \right),&   z\in(-1, \lambda_1 ), \\[.4cm]

  \left( \begin{array}{cc}
                                 1 & 0 \\
                                   e^{2n\phi(z)} & 1 \\
                               \end{array}
                             \right),&   z~\mbox{on lens}, \\[.4cm]

   \left(
                               \begin{array}{cc}
                                 1 &  e^{-2n \phi_+(z)} \\
                                 0 & 1 \\
                               \end{array}
                             \right), &  z\in(-\infty,-1).
\end{array}
\right.
\end{equation*}
\subsection{Global Parametrix}
The global parametrix solves the following approximating RH problem, with  the jump along $(-1, \lambda_1)$:
\begin{description}
\item(a)~~  $N(z)$ is analytic in  $\mathbb{C}\backslash
[-1,\lambda_1]$;
\item(b)~~   \begin{equation}\label{Jump-N} N_{+}(x)=N_{-}(x)\left(
       \begin{array}{cc}
       0 & 1 \\
       -1 & 0 \\
       \end{array}
       \right),\quad x\in  (-1,\lambda_1);\end{equation}
\item(c)~~    \begin{equation}\label{NinfiniytOT} N(z)= I+O(z^{-1}) ,\quad z\rightarrow\infty .\end{equation}
  \end{description}

The solution of the RH problem is constructed explicitly ( see \cite{WuXuZhao} ):
\begin{equation}\label{N-expression}
 N(z) =
\left(
  \begin{array}{cc}
    \frac {\eta(z) + \eta^{-1}(z)} {2}&\frac {\eta(z) -\eta^{-1}(z)} {2i} \\
    -\frac {\eta(z) - \eta^{-1}(z)} {2i} &\frac {\eta(z) + \eta^{-1}(z)} {2} \\
  \end{array}
\right), \quad
 \eta(z)=\left ( \frac {z-\lambda_1}{z+1} \right )^{1/ 4},
  \end{equation}
 where the branch is chosen such that   $\eta(z)$ is analytic in $\mathbb{C}\setminus [-1, \lambda_1]$, and $\eta(z)\sim 1$ as $z\to\infty$.

\subsection{Local parametrix near $z=1$}

The jump matrices for $S(z)$ are not close to the identity matrix near the node points $z=\pm 1$. Thus,  local parametrices have to be constructed in the neighborhoods of $z=\pm 1$. Near  $z = -1$, the parametrix $P^{(-1)}(z)$ can be constructed in terms of the Airy function \cite{d,DeiftZhouS}.  We proceed to find a local parametrix $P^{(1)}(z)$  in $U(1,r)$,  which is an open disc centered at   $z=1$ with radius $r>0$. The parametrix solves  the following RH problem:

\subsection*{Riemann-Hilbert problem for $P^{(1)}$}
\begin{description}
  \item(a)~~ $P^{(1)}(z)$  is analytic in $U(1,r) \backslash  \Sigma_{S}$;    \item(b)~~ On $\Sigma_{S}\cap U(1,r)$, $P^{(1)}(z)$ satisfies the same jump condition as $S(z)$,
  \begin{equation}\label{PJS}
P^{(1)}_+(z)=P^{(1)}_-(z)J_{S}, ~~z\in\Sigma_S \cap U(1,r);\end{equation}
  \item(c)~~  $P^{(1)}(z)$ satisfies the following matching condition on $\partial U(1,r)$:
\begin{equation}\label{mathcing condition}
P^{(1)}(z)N^{-1}(z)=I+ O\left (n^{-1/3}\right );
 \end{equation}
 \item(d)~~  The behavior of  $P^{(1)}(z)=O(\ln(z-\lambda_k) $ as $z\to \lambda_k$ for $k=1,2$.
 \end{description}

To construct  the local  parametrix, we introduce the following model RH  problem, which shares the same jump condition as
$P^{(1)}(z)e^{-n\phi(z)\sigma_3}$.

\subsection*{The  Riemann-Hilbert problem for $\Psi$}
 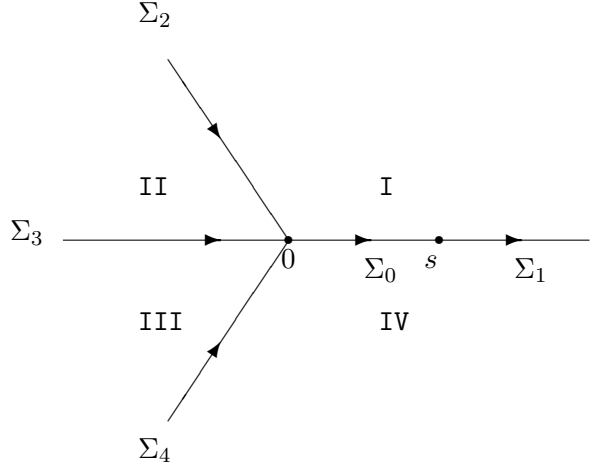
\begin{figure}[t]
\begin{center}
   \setlength{\unitlength}{1truemm}
   \begin{picture}(80,70)(-5,2)
       \put(40,40){\line(-2,-3){16}}
       \put(40,40){\line(-2,3){16}}
       \put(40,40){\line(-1,0){30}}
       \put(40,40){\line(1,0){40}}

       \put(30,55){\thicklines\vector(2,-3){1}}
       \put(30,40){\thicklines\vector(1,0){1}}
       \put(50,40){\thicklines\vector(1,0){1}}
       \put(70,40){\thicklines\vector(1,0){1}}
       \put(30,25){\thicklines\vector(2,3){1}}

       \put(39,36.3){$0$}
        \put(58,36.3){$s$}
       \put(20,11){$\Sigma_4$}
       \put(20,69){$\Sigma_2$}
       \put(3,40){$\Sigma_3$}
        \put(70,35){$\Sigma_1$}
         \put(50,35){$\Sigma_0$}

       \put(52,46){$\texttt{I}$}
       \put(52,28){$\texttt{IV}$}
       \put(20,46){$\texttt{II}$}
       \put(20,28){$\texttt{III}$}

       \put(40,40){\thicklines\circle*{1}}
       \put(60,40){\thicklines\circle*{1}}
  \end{picture}
   \caption{The jump contours and regions for the RH problem for $\Psi$ for $s>0$.}
   \label{fig:Psi}
\end{center}
\end{figure}
\begin{description}
  \item(a)~~  $\Psi(\zeta;x,s)$ ($\Psi(\zeta)$, for short) is analytic in
  $\mathbb{C}\backslash \bigcup_{j=0}^4\Sigma_j$,
  where  the jump contours are indicated in Fig. \ref{fig:Psi};

  \item(b)~~  $\Psi(\zeta)$  satisfies the jump condition for $s>0$
  \begin{gather}\label{Psi-jump}
  \Psi_+ (\zeta)=\Psi_- (\zeta)
  \left\{ \begin{array}{ll}
            \left(
                               \begin{array}{cc}
                                 1 & \omega_2 \\
                                 0 & 1 \\
                                 \end{array}
                             \right), & \zeta \in (s,+\infty), \\[.4cm]
            \left(
                               \begin{array}{cc}
                                 1 & \omega_1 \\
                                 0 & 1 \\
                                 \end{array}
                             \right), & \zeta \in (0,s), \\[.4cm]
           \left(
                               \begin{array}{cc}
                                 1 & 0 \\
                                 1 & 1 \\
                                 \end{array}
                             \right), & \zeta \in {\Sigma}_2,\\[.4cm]
           \left(
                               \begin{array}{cc}
                                 0 & 1 \\
                                 -1 & 0 \\
                                 \end{array}
                             \right),&
                                           \zeta \in {\Sigma}_3,\\  [.4cm]
            \left(
                               \begin{array}{cc}
                                 1 & 0 \\
                                 1 & 1 \\
                                 \end{array}
                             \right),  & \zeta \in \Sigma_4;
          \end{array}
    \right .
  \end{gather}

\item(c)~~     As $\zeta\rightarrow \infty$,
\begin{equation}\label{psi-infinity}
\Psi(\zeta)= \left( \begin{array}{cc} 1 & 0\\ ir(x,s)   & 1 \end{array}\right)
      \left [I+\frac {\Psi_{1}(x,s)}{\zeta}
     +O\left (  \zeta^{-2}\right )\right]\zeta^{-\frac{1}{4}\sigma_3} \frac{I+i\sigma_1}{\sqrt{2}} e^{-(\frac{2}{3}\zeta^{3/2}+x\zeta^{1/2}) \sigma_3} ,
\end{equation}
  where $r(x,s)=i (\Psi_1(x,s))_{12}$.
\item(d)~~ As $\zeta\rightarrow 0$,
 \begin{equation}\label{Psi0}
\Psi(\zeta)= \Psi^{(0)}(\zeta)\left (I+ \frac{1-\omega_1}{2\pi i}\left(\begin{array}{cc} 0 & 1\\ 0 & 0\end{array}\right)\ln \zeta \right ) E,
\end{equation}
where    $\Psi^{(0)}(\zeta)$ is analytic at $\zeta=0$ with the expansion
 \begin{equation}\label{def: Phat}\Psi^{(0)}(\zeta)=\hat{P_0}(x,s)(I+\hat{P}_1(x,s)\zeta+O(\zeta^2)) .\end{equation}
And the piecewise constant matrix
\begin{equation*}
{E}=\left\{\begin{array}{ll}
 \left(\begin{array}{cc} 1 & 0\\ 0 & 1\end{array}\right), & \zeta\in \Omega_1,\\
 \left(\begin{array}{cc} 1 & 0\\ -1 & 1\end{array}\right),  & \zeta\in \Omega_2,\\
  \left(\begin{array}{cc} 1-\omega_1 & -\omega_1\\ 1 & 1\end{array}\right),  & \zeta\in \Omega_3,\\
  \left(\begin{array}{cc} 1 & -\omega_1\\ 0 & 1\end{array}\right),  & \zeta\in \Omega_4.
 \end{array}\right.
 \end{equation*}
\item(e)~~ As $\zeta\rightarrow s$,
 \begin{equation}\label{Psis}
\Psi(\zeta)= \Psi^{(1)}(\zeta)\left (I+ \frac{\omega_1-\omega_2}{2\pi i}\left(\begin{array}{cc} 0 & 1\\ 0 & 0\end{array}\right)\ln (\zeta-s)\right ) \widehat{E},
\end{equation}
where  $\Psi^{(1)}(\zeta)$  is analytic at $\zeta=s$ with the following expansion
 \begin{equation}\label{def: Qhat}\Psi^{(1)}(\zeta)=\hat{Q}_0(x,s)(I+\hat{Q}_1(x,s)(\zeta-s)+O((\zeta-s)^2)).\end{equation}
Here, the piecewise constant matrix
 \begin{equation*} \widehat{E}=\left\{\begin{array}{ll}
 \left(\begin{array}{cc} 1 & 0\\ 0 & 1\end{array}\right), & \Im \zeta>0,\\
 \left(\begin{array}{cc} 1 & -\omega_2\\ 0 & 1\end{array}\right),  & \Im \zeta<0.
 \end{array}\right.
 \end{equation*}
\end{description}

The RH problem for $\Psi$ appears recently  in the studies of the Fredholm determinants of Painlev\'e II  kernel and
 Painlev\'e XXXIV kernel in \cite{XuDai2019} by the second author of the present work and Dai. It also arises in the studies of   the determinants of the  Airy kernel with several discontinuities in \cite{cd} by Claeys and Doeraene, when the number of discontinuities therein equals to two.  The existence of solution to the RH problem for $\Psi$  is  proved.  It is also shown  that $\Psi(\zeta;x,s)$ satisfies the following Lax pair
\begin{gather}
\Psi_{\zeta}(\zeta;x,s)=\left(\begin{array}{cc} \frac{v_{1x}}{2\zeta}+\frac{v_{2x}}{2(\zeta-s)} & i-\frac{iv_{1}}{\zeta}-\frac{iv_{2}}{\zeta-s}\\
-i\left(\zeta+x+v_1+v_2+\frac{v_{1x}^2}{4v_1\zeta}+\frac{v_{2x}^2}{4v_2(\zeta-s)}\right) & -\frac{v_{1x}}{2\zeta}-\frac{v_{2x}}{2(\zeta-s)}\end{array}\right) \Psi(\zeta;x,s),\label{Lax pair-z}\\
\Psi_{x}(\zeta;x,s)=\left(\begin{array}{cc} 0 & i\\ -i\zeta-2i(v_1+v_2+\tfrac{x}{2}) & 0\end{array}\right)\Psi(\zeta;x,s)\label{eq: PsiLax pair}.
  \end{gather}
The compatibility condition of the Lax pair is described by the coupled Painlev\'e II system \eqref{eq:CPII}.
Moreover, the Hamiltonian \eqref{def:Hamiltonian-CPII}
is related to the coefficient of $1/\zeta$ in the large-$\zeta$ expansion of $\Psi(\zeta)$ in \eqref{psi-infinity}
by
 \begin{equation}\label{def:r}
H_{\texttt{II}}(x;s)=\frac{x^2}{4}+ r(x,s);
 \end{equation}
 see \cite[Equation (4.21)]{XuDai2019}.

We introduce the conformal mapping
\begin{equation}\label{f}
f(z)=\left (\frac{3}{2}\phi(z)\right )^{2/3}=2(z-1)+\frac{1}{5}(z-1)^2+O\left ((z-1)^3\right ),
\end{equation} from a neighborhood of $z=1$ to that of the origin.
Then the local parametrix  ${P}^{(1)}(z)$ can be constructed  for   $z\in  U(1, r)$ as follows
\begin{equation}\label{parametrix}
P^{(1)}(z)=E(z)\Psi\left (n^{2/3}(f(z)-f(\lambda_1));n^{2/3}f(\lambda_1),n^{2/3}(f(\lambda_2)-f(\lambda_1))\right )e^{n\phi(z)\sigma_3},\end{equation}
and the pre-factor
\begin{equation}\label{E}
E(z)=N(z)\frac{1}{\sqrt{2}}(I-i\sigma_1)
\left [n^{2/3}(f(z)-f(\lambda_1))\right ]^{\sigma_3/4}\left( \begin{array}{cc} 1 & 0\\ -iH_{\texttt{II}}(n^{2/3}f(\lambda_1); n^{2/3}(f(\lambda_2)-f(\lambda_1)))   & 1 \end{array}\right),\end{equation}
where $H_{\texttt{II}}(x;t)$ is the Hamiltonian given in \eqref{def:Hamiltonian-CPII} and  related to the coefficient  of $1/\zeta$ in the large-$\zeta$ expansion of $\Psi(\zeta)$ by \eqref{def:r}.

\begin{pro} \label{pro: parametrix} For
$$\lambda_k=\frac{s_k}{\sqrt{2n}}=1+\frac {t_k}{2n^{2/3}}$$
 with bounded real parameters $t_k$, $k=1,2$,
 the local parametrix defined in \eqref{parametrix}  and \eqref{E} solves the RH problem for $P^{(1)}$. Moreover, we have the
 expansion   for  $z\in \partial U(1, r)$:
 \begin{equation}\label{R-jump}
P^{(1)}(z) N(z)^{-1}=I+\frac{\Delta(z)}{n^{1/3}}+O(n^{-2/3}),
\end{equation}
where
\begin{equation}\label{def:Delta}
\Delta(z)=\frac{H_{\texttt{II}}(n^{2/3}f(\lambda_1); n^{2/3}(f(\lambda_2)-f(\lambda_1)))}{ 2(f(z)-f(\lambda_1))^{1/2}}N(z)(\sigma_3-i\sigma_1)N^{-1}(z)=O(1),
\end{equation}
and $H_{\texttt{II}}(x,t)$ is the Hamiltonian defined in \eqref{def:Hamiltonian-CPII}.
\end{pro}
\begin{proof}
Taking the principle branch for the fractional power,  it follows from  \eqref{f} that
$$(f(x)-f(\lambda_1))^{\frac{1}{4}}_+=(f(x)-f(\lambda_1))^{\frac{1}{4}}_{-}~e^{\frac{\pi}{2}i}, \quad x<\lambda_1.$$
This, together with the expression of $N(z)$ in \eqref{N-expression},
 implies that $E(z)$ is analytic  for z  in $U(1, r)$.  Recalling the properties of $\Psi(\zeta)$ in \eqref{Psi-jump}-\eqref{Psis}, it is  then seen that the jump condition and the local behaviors near $\lambda_k, k=1,2$,   in the Riemann-Hilbert problem for $P^{(1)}$,  are fulfilled.

     We then proceed to check the matching condition \eqref{mathcing condition}.  Substituting the large-$\zeta$ behavior of $\Psi(\zeta)$ \eqref{psi-infinity} into  \eqref{parametrix}  leads us to  the expansion for $z$ on $\partial U(1, r)$ as $n\to \infty$:
   \begin{align}\label{eq:P1Exp}
P^{(1)}(z)N(z)^{-1} &=N(z)\frac{1}{\sqrt{2}}(I-i\sigma_1)
\left [n^{2/3}(f(z)-f(\lambda_1))\right ]^{\sigma_3/4}\left( \begin{array}{cc} 1 & 0\\ -\frac{i}{4}(n^{2/3}f(\lambda_1))^2  & 1 \end{array}\right)\nonumber\\
&~~\left(I+\frac{\Psi_1(n^{2/3}f(\lambda_1),n^{2/3}(f(\lambda_2)-f(\lambda_1)))}{n^{2/3}(f(z)-f(\lambda_1))}+O(n^{-4/3})\right)\nonumber\\
&~~\left [n^{2/3}(f(z)-f(\lambda_1))\right ]^{-\sigma_3/4}\frac{1}{\sqrt{2}}(I+i\sigma_1)e^{n\rho(z;\lambda_1)\sigma_3}N(z)^{-1}
,
\end{align}
    where
\begin{align}\label{def:rho}
\rho(z;\lambda_1)&=\frac{2}{3}f(z)^{3/2}-\frac{2}{3}(f(z)-f(\lambda_1))^{3/2}-f(\lambda_1)(f(z)-f(\lambda_1))^{1/2}\nonumber\\
&=\frac{1}{\left(f(z)-f(\lambda_1)\right)^{1/2}}\left(\frac{2}{3}f(z)^{2}\left(1-\frac{f(\lambda_1)}{f(z)}\right)^{1/2}-\frac{2}{3}f(z)^{2}\left(1-\frac{f(\lambda_1)}{f(z)}\right)^{2}\right. \nonumber\\
& ~~~~~~~~~~~~~~~~~~~~~~~~~~~~~~~~~\left.-f(z)f(\lambda_1)\left(1-\frac{f(\lambda_1)}{f(z)}\right)\right)\nonumber\\
&=\frac{f(\lambda_1)^2}{4\left(f(z)-f(\lambda_1)\right)^{1/2}}(1+O(f(\lambda_1)).
\end{align}
 Here $f(\lambda_1)\sim 2( \lambda_1-1)$ as $\lambda_1\to 1$.
Inserting the definition of $\Psi_1$ in \eqref{psi-infinity} and \eqref{def:rho} into \eqref{eq:P1Exp}, we obtain \eqref{R-jump}.
 For $z\in \partial U(1, r)$, the denominator  in \eqref{def:Delta}, namely $f(z)-f(\lambda_1)$,  is bounded away from zero.  It follows from Proposition \ref{pro:H} and \eqref{eq: DHaml}, $H(x;s)$ is analytic for real variables $x$ and $s$ and thus bounded. Therefore,  the factor $\Delta$ defined in \eqref{def:Delta} is bounded for  $z\in \partial U(1, r)$. Thus, we obtain the matching condition \eqref{mathcing condition} and complete the proof of Proposition \ref{pro: parametrix}.
\end{proof}

\begin{rem} \label{rem:parameters}
The estimate    in \eqref{def:Delta} and thus the
 matching condition  \eqref{mathcing condition} can be established  for more general parameters:
 $$\lambda_k=\frac{s_k}{\sqrt{2n}}=1+\frac {t_k}{2n^{2/3}}$$
 where
 \begin{equation}\label{def:t}-c_1\leqslant t_1< t_2\leqslant c_2 n^{1/6}, \qquad  t_2-t_1\leqslant c_3, \end{equation}
 for any given positive constants $c_k$ , $k=1,2,3$.
Actually, for such parameters, we have
$$n^{2/3}(f(\lambda_2)-f(\lambda_1))=t_2-t_1+O(n^{-1/3}).$$  In view of the asymptotic behavior \eqref{eq:u-asy} and the relation \eqref{eq: DHaml},    we know that $H_{\texttt{II}}(x;s)$ is exponentially small  for bounded $s$ and large positive $x$.
 Therefore, we have the estimate   \eqref{def:Delta}  for  $ d \leqslant t_1\leqslant c_2 n^{1/6}$ with a certain big enough constant $d$.
This, together with  the estimate  \eqref{def:Delta}  derived  before for  $ -c_1\leqslant t_1\leqslant d$, leads us to the claim.
\end{rem}
\subsection{The final transformation: $S\rightarrow R$}
The final transformation is defined by
\begin{equation}\label{S-R}
R(z)=\left\{ \begin{array}{ll}
                S(z)N^{-1}(z), & z\in \mathbb{C}\backslash \left \{ U(-1,r)\cup U(1,r)\cup \Sigma_S \right \},\\
               S(z) \left\{P^{(-1)}(z)\right\}^{-1}, & z\in   U(-1,r)\backslash \Sigma_{S},  \\
               S(z)  \left\{P^{(1)}(z)\right\}^{-1}, & z\in   U(1,r)\backslash
               \Sigma_{S} .
             \end{array}\right .
\end{equation}
From the matching condition \eqref{mathcing condition}, we have
\begin{equation}
\|J_R(z)-I\|_{L^2\cap L^{\infty}(\Sigma_R)}=O(n^{-1/3}),
\end{equation}
where the error bound is uniform for the parameters $t_1$ and $t_2$ specified by \eqref{def:t}.
Thus, by a standard argument as given in \cite{d, DeiftZhouU, DeiftZhouS}, we have the estimate
 \begin{equation}\label{R-asymptotic}R(z)=I+O(n^{-1/3}),
\end{equation}  where the error bound is uniform for $z$ in whole complex plane.

\section{Proofs of Theorem \ref{thm:HankelAsy}-\ref{thm:OpAsy}}\label{ProofsofTherorem}

In this section, we will prove the main results on the asymptotics of the Hankel determinants and several  quantities related to the orthogonal polynomials,  including the recurrence coefficients and the leading coefficients. Moreover, we will derive the asymptotics of the coupled Painlev\'e IV system.
\subsection{Proof of Theorem \ref{thm:HankelAsy}: asymptotic of the  Hankel determinants}

\begin{lem}\label{thm:F-asy}
Let
$$s_k=\sqrt{2n}+\frac {t_k}{2n^{1/6}}, \quad k=1,2, $$
and $F(s_1,s_2)$ be the logarithmic derivative of the Hankel determinant defined in  \eqref{def:F}, we have
\begin{equation}\label{F-asy}
   F(s_1,s_2)=\sqrt{2}n^{1/6}H_{\texttt{II}}(t_1; t_2-t_1)+O(n^{-1/6}), \end{equation}
where  $H_{\texttt{II}}(x;s)$ is the Hamiltonian for the coupled Painlev\'e II system as defined in \eqref{def:Hamiltonian-CPII}. The error bound is uniform for $-c_1\leqslant t_1< t_2\leqslant c_2 n^{1/6}$  and $  t_2-t_1\leqslant c_3$
 for any given positive constants $c_k$ , $k=1,2,3$; see also \eqref{def:t}.
\end{lem}
\begin{proof}
Tracing back the series of invertible transformations \eqref{Y to T}, \eqref{T-S} and \eqref{S-R}
$$Y\to T\to S\to R, $$
we have
\begin{equation}\label{Psi-Y}
Y_{+}(\sqrt{2n}z)=(2n)^{\frac 12n\sigma_3}e^{\frac 12 nl\sigma_3}R(z)E(z)\Psi_{+}\left (n^{2/3}(f(z)-f(\lambda_1)); x_n,s_n\right )e^{nz^2\sigma_3},  \quad  \lambda_1<z<1+r,
\end{equation} 
where $E(z)$ as defined in \eqref{E} is analytic for $|z-1|<r$. With the parameters specified by \eqref{def:t}, we have
\begin{equation}\label{eq:xExp}
x_n=n^{2/3}f(\lambda_1) = t_1+O(n^{-1/3}),  \quad s_n=n^{2/3}(f(\lambda_2)-f(\lambda_1))=t_2-t_1+O(n^{-1/3}).\end{equation}
Thus, substituting \eqref{Psi-Y} and the estimate \eqref{R-asymptotic}
 into the differential identity \eqref{def: DiffId-1}, we obtain
\begin{align}
    F(s_1,s_2)
    =\frac{1-\omega_1}{\sqrt{2}\pi i}n^{1/6}
    \left(\Psi^{-1}\Psi_{\zeta}
    \right)_{21}(0)
    +\frac{\omega_1-\omega_2}{\sqrt{2}\pi i}n^{1/6}
    \left(\Psi^{-1}\Psi_{\zeta}
    \right)_{21}(s_n)+O(n^{-1/6}),
\end{align}
where the error bound is uniform for $s_1$ and $s_2$ specified by \eqref{def:t}.
 Using the expansions of $\Psi(z)$ near $z=0$ and $z=s$ in \eqref{Psi0} and \eqref{Psis},  we have
\begin{align}\label{eq: F-PQ}
    F(s_1,s_2)
    =\frac{1-\omega_1}{\sqrt{2}\pi i}n^{1/6}
    (\hat{P}_1)_{21}(x_n,s_n)
    +\frac{\omega_1-\omega_2}{\sqrt{2}\pi i}n^{1/6}
    (\hat{Q}_1)_{21}(x_n,s_n)+O(n^{-1/6}),
\end{align}
where   $\hat{P}_1$ and  $\hat{Q}_1$ are defined in \eqref{def: Phat} and \eqref{def: Qhat}, respectively.

Next, we express $\hat{P}_1$ and  $\hat{Q}_1$  in terms of the coupled Painlev\'e II system. Applying the differential equation \eqref{Lax pair-z}, we obtain
\begin{gather}\label{}
\frac{1-\omega_1}{2\pi i}\hat{P}_0\left(\begin{array}{cc} 0 & 1\\ 0 & 0\end{array}\right)\hat{P}_0^{-1}=
\left(\begin{array}{cc} \frac{v_{1x}}{2} & -iv_1\\ -i\frac{v_{1x}^2}{4v_1} & -\frac{v_{1x}}{2}\end{array}\right),\label{hatP0}\\
\hat{P}_1+\frac{1-\omega_1}{2\pi i}\left[\hat{P}_1,\left(\begin{array}{cc} 0 & 1\\ 0 & 0\end{array}\right)\right]
= \hat{P}_0^{-1}
 \left(\begin{array}{cc} -\frac{v_{2x}}{2s} & i+\frac{iv_{2}}{s}\\
-i\left(x+v_1+v_2-\frac{v_{2x}^2}{4v_2s}\right) &\frac{v_{2x}}{2s}\end{array}\right) \hat{P}_0. \label{hatP1}
\end{gather}
Now
$$\hat{P}_0\left(\begin{array}{cc} 0 & 1\\ 0 & 0\end{array}\right)\hat{P}_0^{-1}= \left(
                               \begin{array}{cc}
                                 -(\hat{P}_0)_{11} (\hat{P}_0)_{21}& (\hat{P}_0)_{11}^2 \\
                                 -(\hat{P}_0)_{21}^2 & (\hat{P}_0)_{11} (\hat{P}_0)_{21} \\
                                 \end{array}
                             \right).$$
Then, this, together with \eqref{hatP0},  leads us to
\begin{equation}\label{eq:P0ij}\frac{1-\omega_1}{2\pi i}  (\hat{P}_0)_{11}^2=-iv_1, \quad  \frac{1-\omega_1}{2\pi i}  (\hat{P}_0)_{21}^2=i\frac{v_{1x}^2}{4v_1},
\quad  \frac{1-\omega_1}{2\pi i}  (\hat{P}_0)_{11}(\hat{P}_0)_{21}=-\frac{v_{1x}}{2}.
\end{equation}
From the $(21)$ entry of the matrix equation \eqref{hatP1}, it is seen that
 \begin{equation}\label{eq:P112}
(\hat{P}_1)_{21}=-i\left(x+v_1+v_2-\frac{v_{2x}^2}{4v_2s}\right)(\hat{P_0})_{11}^2
+\frac{v_{2x}}{s}(\hat{P}_0)_{11}(\hat{P}_0)_{21}+i\left(1+\frac{v_2}{s}\right)(\hat{P_0})_{21}^2.
\end{equation}

Substituting \eqref{eq:P0ij} into \eqref{eq:P112}, we obtain
\begin{equation}\label{eq:P1}
\frac{1-\omega_1}{\pi i}(\hat{P}_1)_{21}=-2v_1\left(x+v_1+v_2-\frac{v_{2x}^2}{4v_2s}\right)
-\frac{v_{1x}v_{2x}}{s}+\left(1+\frac{v_2}{s}\right)\frac{v_{1x}^2}{2v_1}.
\end{equation}
Similarly, we have that
\begin{equation}\label{eq:Q0ij}\frac{\omega_1-\omega_2}{2\pi i}  (\hat{Q}_0)_{11}^2=-iv_2, \quad  \frac{\omega_1-\omega_2}{2\pi i}  (\hat{Q}_0)_{21}^2=i\frac{v_{2x}^2}{4v_2},
\end{equation}
and
\begin{equation}\label{eq:Q1}
\frac{\omega_1-\omega_2}{\pi i}(\hat{Q}_1)_{21}=-2v_2\left(x-s+v_1+v_2+\frac{v_{1x}^2}{4v_1s}\right)
+\frac{v_{1x}v_{2x}}{s}-\left(1-\frac{v_1}{s}\right)\frac{v_{2x}^2}{2v_2}.
\end{equation}
Therefore, we obtain from \eqref{eq:CPII}, \eqref{eq:xExp} \eqref{eq: F-PQ}, \eqref{eq:P1} and \eqref{eq:Q1} that
\begin{align}
    F(s_1,s_2)
    &=\frac{1}{\sqrt{2}}n^{1/6}\left(-
    2v_2s_n-2(v_1+v_2)(x_n+v_1+v_2)+\frac{v_{1x}^2}{2v_1}+\frac{v_{2x}^2}{2v_2}\right)+O(n^{-1/6})\\
    &=\sqrt{2}n^{1/6}H_{\texttt{II}}(x_n;s_n)+O(n^{-1/6})\\
    &=\sqrt{2}n^{1/6}H_{\texttt{II}}(t_1; t_2-t_1)+O(n^{-1/6}),
\end{align}
where $H_{\texttt{II}}(x;s)$ is the Hamiltonian for the coupled Painlev\'e II system as defined in \eqref{def:Hamiltonian-CPII}.
This completes the proof Lemma \ref{thm:F-asy}.
\end{proof}

Next we derive the asymptotic expansion for the Hankel determinant $D_n$ defined by \eqref{def: Hankel} when the jump discontinuities of the weight function \eqref{weight} are large enough.
\begin{lem}\label{lem:Hankel-asy-out}
For $s_2>s_1\geqslant\sqrt{2n}+c_0$ and  any given positive constant $c_0$,
we have the asymptotic approximation for the Hankel determinant $D_n(s_1,s_2)=D_n(s_1,s_2;\omega_1,\omega_2)$ defined by \eqref{def: Hankel}
\begin{equation}\label{D-asy}
   D_n(s_1,s_2)=D_n^{\texttt{GUE}}\left(1+O\left(e^{-cn^{1/4}}\right)\right), \end{equation}
where $c$ is some positive constant  and $D_n^{\texttt{GUE}}$, given explicitly in \eqref{def: Hankel-GUE},  is the Hankel determinant associated with the pure Gaussian weight.
\end{lem}

\begin{proof}

For $s_2>s_1\geqslant\sqrt{2n}+c_0$, we have $\lambda_2>\lambda_1\geqslant 1+\frac{c_0}{\sqrt{2n}} >1$.
On account of  \eqref{f},  there is some constant $c>0$ such that
$$n\phi(\lambda)>cn^{1/4} $$
for  $\lambda>\lambda_1$.
Thus, the jump matrices $J_S(z)$ defined in \eqref{def:SJump-1} tend to the identity matrix exponentially fast for
$z\in (\lambda_1, \lambda_2) \cup ( \lambda_2, +\infty)$.
Therefore, we have
\begin{equation}S(z)=(I+O(e^{-2n\phi(\lambda_1)})S_0(z), \end{equation}
where $S_0(z)$ is solution to the  RH problem for $S$ when the parameters $\omega_1=\omega_2=1$ and $\lambda_1=1$ therein.
Tracing back the sequence of transformations $Y\to T\to S$, given in  \eqref{Y to T}, \eqref{T-S}, we have
\begin{equation}\label{Y-est}
Y(\sqrt{2n}z)=(2n)^{\frac 12n\sigma_3}e^{\frac 12 nl\sigma_3}(I+O(e^{-2n\phi(\lambda_1)})S_0(z)e^{ng(z)\sigma_3-\frac 12 nl \sigma_3},
\end{equation}
where $e^{ng(z)\sigma_3}=(I+
O\left (\frac 1 {z^2}\right ))z^{n\sigma_3}$. In view of \eqref{Y} and the differential identity \eqref{def: DiffId-2}, we obtain
 \begin{equation}
    F(s_1,s_2)=2p_n+O\left(e^{-2n\phi(\lambda_1)}\right)=O\left(e^{-2n\phi(\lambda_1)}\right),
\end{equation}
where $p_n=0$ is the sub-leading coefficient of the monic Hermite polynomial of degree $n$.
We  integrate  on both sides of the above equation and obtain
 \begin{equation}\label{eq:logD-out}
\ln D_n(s_1,s_2)-\ln D_n(s_1+L,s_2+L)=\int_{s_1}^{L+s_1}O\left(e^{-2n\phi(\lambda_1)}\right)d\tau.
\end{equation}
Let $L\to +\infty$, we get \eqref{D-asy} and complete the proof of Lemma \ref{lem:Hankel-asy-out}.
 \end{proof}

 Now, we are ready to prove Theorem \ref{thm:HankelAsy}.
 Integrating on both sides of the  equation \eqref{F-asy}, we obtain for some positive constant $c_0$ that
\begin{equation}\label{eq:log-D-n}
\ln D_n(s_1,s_2)-\ln D_n(s_1+c_0,s_2+c_0)=-\int_{t_1}^{t_0}H_{\texttt{II}}(\tau; t_2-t_1)d\tau+O(n^{-1/6}),
\end{equation}
where 
 $$s_1=\sqrt{2n}+\frac {t_1}{\sqrt{2}n^{1/6}}, \quad s_2=\sqrt{2n}+\frac {t_2}{\sqrt{2}n^{1/6}}$$
for $t_k$, $k=1,2$, in any compact subset of $\mathbb{R}$ and
\begin{equation}\label{eq:t-0}t_0=n^{2/3}f(\frac{s_1+c_0}{\sqrt{2n}})\sim \sqrt{2}c_0n^{1/6} .
\end{equation}
In view of \eqref{D-asy}, we have
 \begin{equation}\label{D-asy-c0}
 \ln D_n(s_1+c_0,s_2+c_0)= D_n^{\texttt{GUE}}\left(1+O\left(e^{-cn^{1/4}}\right)\right), \end{equation}
where $c$ is some positive constant  and $D_n^{\texttt{GUE}}$ is the Hankel determinant associated with the Gaussian weight; see \eqref{def: Hankel-GUE}.
Recalling \eqref{eq: DHaml}, we obtain from an integration by parts that
\begin{equation}\label{eq:int-H}
\int_{t_1}^{t_0}H_{\texttt{II}}(\tau; t_2-t_1)d\tau=\int_{t_1}^{t_0}(\tau-t_1)\left(u_1(\tau; t_2-t_1)^2+u_2(\tau;t_2-t_1)^2\right)d\tau,
\end{equation}
where $u_1(x)$ and $u_2(x)$ are solutions to the coupled nonlinear differential equations  \eqref{int-equation u} subject to the boundary conditions  \eqref{eq:u-asy} as $x\to+\infty$.
On account of  \eqref{eq:u-asy} and  \eqref{eq:t-0}, we have
 \begin{equation}\label{int-rem}
 \int_{t_0}^{+\infty}(\tau-t_1)\left(u_1(\tau; t_2-t_1)^2+u_2^2(\tau;t_2-t_1)\right)d\tau
 =O\left(e^{-cn^{1/4}}\right),\end{equation}
 for some constant $c>0$.  Inserting \eqref{D-asy-c0} and \eqref{int-rem} into \eqref{eq:log-D-n}, we obtain  \eqref{thm: HAsy}.  This completes the proof of   Theorem \ref{thm:HankelAsy}.

\subsection{Proof of theorem \ref{thm:CPIVAsy}: asymptotics of the coupled Painlev\'e IV } 


From this section, the parameters in \eqref{weight} are defined by
$s_1=\sqrt{2n}\lambda_1=\sqrt{2n}+\frac {t_1}{\sqrt{2}n^{1/6}}$ and $ s_2=\sqrt{2n}\lambda_2=\sqrt{2n}+\frac {t_2}{\sqrt{2}n^{1/6}}, $
with $t_1$ and $t_2$  in any compact subset of $\mathbb{R}$.
Tracing back the sequence of transformations $Y\to T\to S\to R$, given in  \eqref{Y to T}, \eqref{T-S}  and \eqref{S-R}, we have the expression  for large $z$:
\begin{equation}\label{Y tracing back}
Y(\sqrt{2n}z)=(2n)^{\frac 12n\sigma_3}e^{\frac 12 nl\sigma_3}R(z)N(z)e^{ng(z)\sigma_3-\frac 12 nl \sigma_3},
\end{equation}
where $l=-1-2\ln2$.
From the definition of $g(z)$ in \eqref{g}, it is seen that
$$e^{ng(z)\sigma_3}z^{-n\sigma_3}=I+
O\left (\frac 1 {z^2}\right ), \quad z\to \infty.$$
By the expression of $N(z)$ in \eqref{N-expression}, we have the expansion
\begin{equation}\label{N-expand}
N(z)=I+\frac{N_1}{z}+O\left (\frac 1 {z^2}\right ),\quad  z\to \infty, \end{equation}
where $N_1=-\frac{1}{4}(1+\lambda_1)\sigma_2$.
It  follows from the expansion of the jump for $R(z)$ in \eqref{R-jump} and \eqref{def:Delta} that
\begin{equation}
R(z) =I+\frac{R^{(1)}(z)}{n^{1/3}}+O\left
(\frac 1 {n^{2/3}}\right ),\quad n\rightarrow\infty.
\end{equation}
Here,   $R^{(1)}$  satisfy the jump relation
\begin{equation}\label{}
R^{(1)}_+(z)-R^{(1)}_-(z)=\Delta(z),
\end{equation}
where  $\Delta(z)$ is  given in \eqref{def:Delta}  and $R^{(1)}(z)=O(1/z)$ for $z$ large.
Applying Cauchy's theorem, it is seen that
\begin{equation}\label{R-1}
 R^{(1)}(z)=\left\{\begin{array}{ll}
                   \frac{\sqrt{1+\lambda_1}H_{\texttt{II}}(n^{2/3}f(\lambda_1);n^{2/3}(f(\lambda_2)-f(\lambda_1)))}{2 \sqrt{f'(\lambda_1)}(z-\lambda_1)}(\sigma_3-i\sigma_1)-\Delta(z), &z\in U(1,r),\\[.3cm]
                    \frac{\sqrt{1+\lambda_1}H_{\texttt{II}}(n^{2/3}f(\lambda_1);n^{2/3}(f(\lambda_2)-f(\lambda_1)))}{2 \sqrt{f'(\lambda_1)} (z-\lambda_1)}(\sigma_3-i\sigma_1), &z\not\in \overline{U(1,r)}.
                  \end{array}
\right.
\end{equation}
Thus, we get from the expression the following expansion as $z\to\infty$
\begin{equation}\label{R-expand}
R(z)=I+\frac{R_1}{z}+O\left (\frac 1 {z^2}\right ),
\end{equation}
and
\begin{align}\label{R1-expand}
R_1&= \frac{\sqrt{1+\lambda_1}}{2 \sqrt{f'(\lambda_1)}n^{1/3} }H_{\texttt{II}}(n^{2/3}f(\lambda_1); n^{2/3}(f(\lambda_2)-f(\lambda_1)))(\sigma_3-i\sigma_1)\nonumber\\
&=\frac{1}{2 n^{1/3}}H_{\texttt{II}}(t_1; t_2-t_1)(\sigma_3-i\sigma_1) +O(n^{-2/3}),
\end{align}
where use is also made of \eqref{f}.
In view of \eqref{N-expression},  \eqref{f} and  \eqref{def:Delta}, we have the following expansion  as $z\to \lambda_1$
\begin{equation}\label{R-1-lambda}
R^{(1)}(z)=-\frac{1}{10}H_{\texttt{II}}(t_1; t_2-t_1)(\sigma_3-i\sigma_1)+O(n^{-2/3})+O(z-\lambda_1).
\end{equation}
Substituting  \eqref{N-expand}, \eqref{R-expand}  and
\eqref{R1-expand} into  \eqref{Y tracing back} yields
\begin{align}
Y_1&=\sqrt{2n}(2n)^{\frac 12n\sigma_3}e^{\frac 12 n l \sigma_3}(R_1+N_1)e^{-\frac 12 n l \sigma_3}(2n)^{-\frac 12n\sigma_3}\\
&=  \sqrt{2n}(2n)^{\frac 12n\sigma_3}e^{\frac 12 n l \sigma_3}\left(-\frac{1}{2}\sigma_2+\frac{H_{\texttt{II}}(t_1;t_2-t_1)}{2n^{1/3} }(\sigma_3-i\sigma_1) +O(n^{-2/3})\right)    e^{-\frac 12 n l \sigma_3}(2n)^{-\frac 12n\sigma_3}. \label{Y1}\end{align}
This, together with  the relation \eqref{def:y}, implies that
\begin{align}\label{asymp-y}
y(x;s)&=-2(Y_1)_{21}e^{-x^2}\nonumber\\
&=i(2n)^{-n+\frac 12}e^{-x^2-nl}\left(1+\frac{H_{II}(t_1; t_2-t_1)}{n^{1/3}}
+O(n^{-2/3})\right),\end{align}
where
$x=\frac{s_1+s_2}{2}=\sqrt{2n}+\frac{t_1+t_2}{2\sqrt{2}n^{1/6}}$ and $ s=\frac{s_2-s_1}{2}=\frac{t_2-t_1}{2\sqrt{2}n^{1/6}}$.
Recalling  $l=-1-2\ln2$, we obtain the  asymptotic expansion of $y(x;s)$ as given in \eqref{asy-cpiv-y}.

Next, we consider the asymptotics of $a_k(x;s)$ and $a_k(x;s)$ for $k=1,2$.
It follows from  the mastar equation of the  Lax pair \eqref{def: Lax pair}  that
\begin{gather}
   a_1(x;s)=\frac 1{ y(x;s)}\lim_{z\rightarrow -s}(z+s)(\Phi_z\Phi^{-1})_{12},\label{a-1-phi}\\
  a_2(x;s)=\frac 1{ y(x;s)}\lim_{z\rightarrow s}(z-s)(\Phi_z\Phi^{-1})_{12},\label{a-2-phi}\\
   b_1(x;s)=\frac 1{ a_1(x;s)}\lim_{z\rightarrow -s}(z+s)(\Phi_z\Phi^{-1})_{11},\label{b-1-phi}\\
    b_2(x;s)=\frac 1{ a_2(x;s)}\lim_{z\rightarrow s}(z-s)(\Phi_z\Phi^{-1})_{11},\label{b-2-phi}
\end{gather}
where the subscript  $z$ denotes  the derivative with respect to $z$.
It is seen from \eqref{eq: Phi} that
\begin{equation}\label{phi-y-asy}
\Phi_z\Phi^{-1}=\sigma_1e^{\frac{x^2}{2}\sigma_3}[Y_z(z+x)Y^{-1}(z+x)
+Y(z+x)(-(z+x)\sigma_3)Y^{-1}(z+x)]
e^{-\frac{x^2}{2}\sigma_3}\sigma_1.
\end{equation}
According to \eqref{a-1-phi} and \eqref{phi-y-asy} and in view of  the fact that $Y(z)$ has at most logarithm singularity at $s_1$, we have
\begin{align}\label{a-1-asymp-o}
  a_1(x;s)&=\frac 1{ y(x;s)}e^{-x^2}\lim_{z\rightarrow -s}(z+s)(Y_z(z+x)Y^{-1}(z+x))_{21},\nonumber\\
  &=\frac 1{ y(x;s)}e^{-x^2}\lim_{z\rightarrow s_1}(z-s_1)(Y_z(z)Y^{-1}(z))_{21},\nonumber\\
  &=\frac 1{ y(x;s)}e^{-x^2}\lim_{z\rightarrow\lambda_1}\sqrt{2n}(z-\lambda_1)
  (Y_z(\sqrt{2n}z)Y^{-1}(\sqrt{2n}z))_{21}.
\end{align}
Inserting   \eqref{Psi-Y} into \eqref{a-1-asymp-o}, it is seen that
\begin{equation}\label{a-1-asymp}
  a_1(x;s)=\frac {(2n)^{-n}}{ y(x;s)}e^{-x^2-nl} \left(R(\lambda_1)E(\lambda_1)\left(\lim_{\zeta\rightarrow 0}\zeta
\Psi_{\zeta}(\zeta) \Psi^{-1}(\zeta)\right)E^{-1}(\lambda_1)R^{-1}(\lambda_1) \right)_{21},
\end{equation}
where use is made of the fact  that $E(z)$ and $R(z)$ are  analytic at $z=\lambda_1$.
It follows from the behavior of $\Psi(z)$ near $z=\lambda_1$ as given in  \eqref{Psi0} that
\begin{equation}\label{eq: Psi-P0}
  \lim_{\zeta\rightarrow 0}\zeta
\Psi_{\zeta}(\zeta) \Psi^{-1}(\zeta)=\frac{1-\omega_1}{2\pi i}\hat{P}_0\left(\begin{array}{cc} 0 & 1\\0 & 0\end{array}\right)\hat{P}_0^{-1}.\end{equation}
From the expression \eqref{E}, we get  for $k=1,2$,
\begin{equation}\label{E-lambda}
E(\lambda_k)=\frac{1}{\sqrt{2}}(I-i\sigma_1)
n^{\sigma_3/6}2^{\sigma_3/2}\left( \begin{array}{cc} 1 & 0\\ -iH_{\texttt{II}}(t_1; t_2-t_1)  & 1 \end{array}\right)(I+O(n^{-2/3})).
\end{equation}
Thus, we obtain from \eqref{R-asymptotic}, \eqref{a-1-asymp}-\eqref{E-lambda} that
\begin{equation}\label{}
  a_1(x;s)=\frac {(1-\omega_1)(2n)^{-n}}{ 2\pi i y(x;s)}e^{-x^2-nl}\nonumber\\
  \left((\hat{P}_0)_{11}^2n^{1/3}+i(\hat{P}_0)_{11}(\hat{P}_0)_{21}+H_{\texttt{II}}(t_1; t_2-t_1) (\hat{P}_0)_{11}^2
 +O(n^{-1/3})\right).
\end{equation}
Using \eqref{eq:P0ij} and \eqref{asymp-y}, we obtain the asymptotic approximation of $a_1(x;s)$ as stated  in \eqref{asy-cpiv-a-1}.

Similarly,   from   \eqref{Psis}, \eqref{parametrix} and  \eqref{a-2-phi},   we have
\begin{align}\label{a-2-asymp-1}
  a_2(x;s)
  &=\frac {(2n)^{-n}}{ y(x;s)}e^{-x^2-nl}\left(R(\lambda_2)E(\lambda_2)\frac{\omega_1-\omega_2}{2\pi i}\hat{Q}_0\left(\begin{array}{cc} 0 & 1\\0 & 0\end{array}\right)\hat{Q}_0^{-1}E^{-1}(\lambda_2)R^{-1}(\lambda_2) \right)_{21}.
\end{align}
Substituting \eqref{R-asymptotic},  \eqref{eq:Q0ij},  \eqref{asymp-y} and  \eqref{E-lambda}  into \eqref{a-2-asymp-1}, we obtain the asymptotic expansion of $a_2(x;s)$  as given in \eqref{asy-cpiv-a-2}.
In view of  \eqref{b-1-phi} and \eqref{b-2-phi}, we obtain  the asymptotics of $b_1(x;s)$ and $b_2(x;s)$ by considering the $(2,2)$ entry of the matrices in \eqref{a-1-asymp} and \eqref{a-2-asymp-1}:
\begin{align}\label{}
  b_1(x;s)
  &=\frac {(1-\omega_1)}{ 2\pi i a_1(x;s)}\left(-i(\hat{P}_0)_{11}^2n^{1/3}
 +O(n^{-1/3})\right)\nonumber\\
  &=\sqrt{2n}\left(1-\frac{v_{1x}(t_1; t_2-t_1)}{2v_{1}(t_1; t_2-t_1)n^{1/3}}+O(n^{-2/3})\right),\quad n\rightarrow\infty,
\end{align}
and
\begin{align}\label{}
  b_2(x;s)
   &=\frac {(\omega_1-\omega_2)}{ 2\pi i a_2(x;s)}\left(-i(\hat{Q}_0)_{11}^2n^{1/3}
 +O(n^{-1/3})\right)\nonumber\\
  &=\sqrt{2n}\left(1-\frac{v_{2x}(t_1; t_2-t_1)}{2v_{2}(t_1; t_2-t_1)n^{1/3}}+O(n^{-2/3})\right),\quad n\rightarrow\infty,
\end{align}
where
$x=\frac{s_1+s_2}{2}=\sqrt{2n}+\frac{t_1+t_2}{2\sqrt{2}n^{1/6}}$ and $ s=\frac{s_2-s_1}{2}=\frac{t_2-t_1}{2\sqrt{2}n^{1/6}}$.
This completes the proof of Theorem  \ref{thm:CPIVAsy}.

\subsection{Proof of Theorem \ref{thm:OpAsy}: asymptotics of the orthogonal polynomials } 

From \eqref{asy-cpiv-a-1}-\eqref{asy-cpiv-b-2}, we have
\begin{equation}\label{eq:abb1}
a_1(x;s)  b_1(x;s)^2=-\sqrt{2}n^{5/6}\left(v_1(t_1; t_2-t_1)-\frac{1}{2}v_{1x}(t_1; t_2-t_1)n^{-1/3}+O(n^{-2/3})\right),
\end{equation}
\begin{equation}\label{eq:abb2}
a_2(x;s)  b_2(x;s)^2=-\sqrt{2}n^{5/6}\left(v_2(t_1; t_2-t_1)-\frac{1}{2}v_{2x}(t_1; t_2-t_1)n^{-1/3}+O(n^{-2/3})\right),
\end{equation}
\begin{equation}\label{eq:ab}
a_1(x;s)  b_1(x;s)+a_2(x;s)  b_2(x;s)=-n^{1/3}\left((v_1(t_1; t_2-t_1)+v_2(t_1; t_2-t_1))+O(n^{-2/3})\right).
\end{equation}
Substituting \eqref{eq:abb1}-\eqref{eq:ab} into \eqref{thm:alpha}, \eqref{thm: beta},  \eqref{thm:pns1} and \eqref{thm:pns2}, we obtain the
asymptotics of the recurrence coefficients and $\pi_n(s_k)$, $k=1,2$,  as given in \eqref{an-asymptotic-expansion}, \eqref{bn-asymptotic-expansion}, \eqref{pns1-asymptotic-expansion} and \eqref{pns2-asymptotic-expansion}, respectively.
 In view of  \eqref{recurrence coefficients and Y},   \eqref{N-expand},
\eqref{R1-expand} and \eqref{Y1}, we derive the asymptotic expansion for the leading coefficient $\gamma_{n-1}$ of the orthonormal polynomial of degree $n-1$:
\begin{align}\label{}
\gamma_{n-1}&=\left(-\frac{1}{2\pi i}(Y_1)_{21}\right)^{1/2}\nonumber\\
&=\left(\frac{\sqrt{2n}(2n)^{-n}e^{-nl}(R_1+N_1)_{21}}{-2\pi i}\right)^{1/2}\nonumber\\
&=\frac{2^{\frac n2-\frac 34}n^{\frac 14-\frac n2}e^{\frac n2}}{\sqrt{\pi}}\left(1+\frac{H_{\texttt{II}}(t_1; t_2-t_1))}{2n^{1/3}}+O(n^{-2/3})\right),\quad n\rightarrow\infty.
\end{align}
Thus, we complete the proof of Theorem \ref{thm:OpAsy}.

\section*{Acknowledgements}
The authors are grateful to Dan Dai for useful comments. Xiao-Bo Wu was partially supported by National Natural Science Foundation
of China under grant number 11801376 and Science Foundation of Education Department of Jiangxi Province  under grant number GJJ170937. Shuai-Xia Xu was partially supported by  National Natural Science Foundation
of China under grant numbers 11971492, 11571376 and 11201493.


\begin{thebibliography}{99}


\bibitem{BC} A. Bogatskiy, T. Claeys and A. Its, Hankel determinant and orthogonal polynomials for a
Gaussian weight with a discontinuity at the edge, {\it  Comm. Math. Phys.},  {\bf 347}, 127--162 (2016).

\bibitem{BP04}  O. Bohigas and M.P. Pato, Missing levels in correlated spectra, {\it Phys. Lett. B }, {\bf 595 }, 171--176 (2004).




\bibitem{BP06} O. Bohigas and M.P. Pato, Randomly incomplete spectra and intermediate statistics, {\it Phys. Rev. E (3)}, {\bf74}, 036212 (2006).

\bibitem{bf} F. Bornemann, P. J. Forrester and A. Mays, Finite size effects for spacing distributions
in random matrix theory: circular ensembles and Riemann zeros, {\it Stud. Appl. Math. },  {\bf 138}, 401--437 (2017).




\bibitem{cc} C. Charlier and T. Claeys, Thinning and conditioning of the Circular Unitary
Ensemble, {\it Random Matrices Theory Appl.}, {\bf 6},   no. 2, 1750007 (2017).

\bibitem{cc-19} C. Charlier and T. Claeys,  Large Gap Asymptotics for Airy Kernel Determinants with Discontinuities, {\it  Commun. Math. Phys.}, https://doi.org/10.1007/s00220-019-03538-w (2019).

\bibitem{cd-1} C. Charlier and A. Dea\~{n}o,
Asymptotics for Hankel determinants associated to a Hermite weight with a varying discontinuity,
    {\it SIGMA Symmetry Integrability Geom. Methods Appl.}, {\bf 14}, Paper No. 018, 43 pp (2018).

\bibitem{cd} T.  Claeys and A. Doeraene, The generating function for the Airy point process and a system of coupled Painlev\'e II equations, {\it Stud. Appl. Math.}, {\bf 140}, 403--437 (2018).

\bibitem{d} P. Deift,  {\it Orthogonal polynomials and random matrices:
a Riemann-Hilbert approach}, Courant Lecture Notes 3, New York
University, Amer. Math. Soc., Providence, RI,  1999.


\bibitem{DeiftZhouU}   P. Deift, T. Kriecherbauer, K. T.-R. McLaughlin, S. Venakides and X. Zhou, Uniform asymptotics for polynomials orthogonal with respect to varying exponential weights and applications to universality questions in random matrix theory, {\it Comm. Pure Appl. Math.},  {\bf 52}, 1335--1425 (1999).


\bibitem{DeiftZhouS}   P. Deift, T. Kriecherbauer, K. T.-R. McLaughlin, S. Venakides and X. Zhou, Strong asymptotics of orthogonal polynomials with respect to exponential weights, {\it Comm. Pure Appl. Math.},  {\bf 52}, 1491--1552 (1999).

\bibitem{DeiftZhouA}  P. Deift and X. Zhou, A steepest descent method for oscillatory Riemann-Hilbert problems, Asymptotics for the MKdV equation, {\it  Ann. Math.},  {\bf 137}, 295--368 (1993).

\bibitem{FokasBook} A.S. Fokas, A.R. Its, A.A. Kapaev and V.Y. Novokshenov, {\it Painlev\'e Transcendents: The Riemann--Hilbert Approach},
AMS Mathematical Surveys and
Monographs,  128, Amer. Math. Soc., Providence, RI, 2006.


\bibitem{Fokas}  A.S. Fokas, A.R. Its and A.V. Kitaev, The isomonodromy approach to matrix models in $2D$ quantum gravity, {\it  Comm. Math. Phys.}, {\bf 147}, 395--430 (1992).


\bibitem{Forrester} P.J. Forrester, {\it Log-gases and Random Matrices}, London Mathematical Society Monographs Series, 34., Princeton University Press, Princeton, NJ, 2010.


\bibitem{FW}  P.J. Forrester and N.S. Witte, Application of the $\tau$-function theory of Painlev\'e equations to random matrices: PIV, PII and the GUE, {\it  Comm. Math. Phys.},  {\bf 219}, 357--398 (2001).




\bibitem{IK} A. Its and I. Krasovsky, Hankel determinant and orthogonal polynomials for the Gaussian weight with a jump, {\it Integrable Systems and Random
Matrices},  J. Baik et al., eds., Contemp. Math., 458,
Amer. Math. Soc., Providence, RI, 2008, 215--247.
%




\bibitem{jmu} M. Jimbo, T. Miwa and K. Ueno, Monodromy preserving deformation of linear ordinary differential equations with rational coefficients II,  {\it Phys. D},
{\bf 2},  407--448 (1981).



\bibitem{Kaw2017} H. Kawakami, Four-dimensional Painlev\'{e}-type equations associated with ramified linear equations III: Garnier systems and Fuji-Suzuki systerms, {\it SIGMA} , {\bf 13}, 096, 50 pp (2017).

\bibitem{Kaw2018} H. Kawakami, Four-dimensional Painlev\'{e}-type equations associated with ramified linear equations II: Sasano systems, \emph{J. Integrable Syst.},  {\bf 3}, no.1, xyy013, 36 pp (2018).



\bibitem{KawNakSak}  H. Kawakami, A. Nakamura, and H. Sakai, Degeneration scheme of 4-dimensional Painlev\'{e}-type equations, arXiv:1209.3836.


\bibitem{MC} C. Min and Y. Chen,
Painlev\'e transcendents and the Hankel determinants generated by a discontinuous Gaussian weight,
{\it Math. Methods Appl. Sci.}, {\bf  42}, no.1, 301--321 (2019).

\bibitem{Metha} M.L. Mehta, Random Matrices, 3rd ed., Elsevier/Academic Press, Amsterdam, 2004.


\bibitem{O}
        F.W.J. Olver, A.B. Olde Daalhuis, D.W. Lozier, B.I. Schneider, R.F. Boisvert, C.W. Clark, B.R. Miller and B.V. Saunders, eds, NIST Digital Library of Mathematical Functions, http://dlmf.nist.gov/, Release 1.0.21 of 2018-12-15.



\bibitem{TW} C. Tracy and H. Widom, Level-spacing distributions and the Airy kernel, {\it Comm. Math. Phys.},
    {\bf 159}, 151--174 (1994).


\bibitem{WuXuZhao} X.-B. Wu, S.-X. Xu and  Y.-Q. Zhao, Gaussian Unitary Ensemble with Boundary Spectrum Singularity and $\sigma$-Form of the Painlev\'e II Equation, {\it Stud. Appl. Math.},  {\bf 140}, 221--251(2018).

\bibitem{XuDai2019} S.-X. Xu and D. Dai,  Tracy-Widom distributions in critical unitary random matrix ensembles and the coupled Painlev\'e II system,
{\it Comm. Math. Phys.},   {\bf 365} no. 2, 515--567 (2019).

\bibitem{XuZhao} S.-X. Xu and  Y.-Q. Zhao, Painlev\'e XXXIV asymptotics of orthogonal polynomials for the Gaussian weight with a jump at the edge, {\it Stud. Appl. Math.},  {\bf 127}, 67--105 (2011).

%
%




\end{thebibliography}
\end{document}